\documentclass[11pt,twoside]{amsart}
\usepackage{amsmath,amscd,amsthm,amssymb,amsxtra,latexsym,epsfig,epic,graphics,mathdots,amssymb,enumitem}
\usepackage[all]{xy}
\SelectTips{cm}{11}
\usepackage[utf8]{inputenc}
\usepackage[svgnames]{xcolor}
\usepackage{tcolorbox}
\tcbuselibrary{skins,listings,theorems}
\usepackage{hyperref}

\usepackage{mathrsfs}
\usepackage{dsfont}

\usepackage{graphicx,color}

\newtheorem{thm}{Theorem}[section]
\newtheorem{mainthm}{Theorem}
\newtheorem*{cor*}{Corollary}
\newtheorem{question}{Question}
\newtheorem{lem}[thm]{Lemma}

\newtheorem{prop}[thm]{Proposition}
\newtheorem{step}{Step}
\newtheorem{claim}{Claim}
\theoremstyle{definition}
\newtheorem{defn}[thm]{Definition}

\newtheorem*{ack}{Acknowledgment}      

\newcommand{\GG}{{\mathds G}}

\newcommand{\NN}{{\mathds N}}
\newcommand{\QQ}{{\mathds Q}}
\newcommand{\PP}{{\mathds {P}}}

\newcommand{\ZZ}{{\mathds Z}}

\newcommand{\sA}{{\mathcal A}}
\newcommand{\sB}{{\mathcal B}}
\newcommand{\sC}{{\mathcal C}}
\newcommand{\sD}{{\mathcal D}}
\newcommand{\sE}{{\mathcal E}}
\newcommand{\sF}{{\mathcal F}}
\newcommand{\sG}{{\mathcal G}}
\newcommand{\sH}{{\mathcal H}}
\newcommand{\sI}{{\mathcal I}}
\newcommand{\sK}{{\mathcal K}}

\newcommand{\sN}{{\mathcal N}}
\newcommand{\sO}{{\mathcal O}}
\newcommand{\sQ}{{\mathcal Q}}
\newcommand{\sP}{{\mathcal P}}
\newcommand{\sU}{{\mathcal U}}
\newcommand{\sR}{{\mathcal R}}
\newcommand{\sS}{{\mathcal S}}

\newcommand{\sV}{{\mathcal V}}
\newcommand{\sY}{{\mathcal Y}}
\newcommand{\sW}{{\mathcal W}}
\newcommand{\sX}{{\mathcal X}}
\newcommand{\sZ}{{\mathcal Z}}

\newcommand{\rH}{{\mathrm H}}
\newcommand{\rR}{{\mathrm R}}
\newcommand{\rh}{{\mathrm h}}
\newcommand{\rp}{{\mathrm p}}
\newcommand{\ru}{{\mathrm u}}
\newcommand{\rM}{{\mathrm M}}
\newcommand{\rt}{{\mathrm t}}
\newcommand{\rD}{{\mathrm D}}

\newcommand{\fh}{{\mathfrak {h}}}
\newcommand{\fl}{{\mathfrak {l}}}

\newcommand{\mono}{\hookrightarrow}
\newcommand{\epi}{\twoheadrightarrow}

\newcommand{\kk}{{\mathds k}}

\newcommand{\rank}{\operatorname{rk}}
\newcommand{\Hom}{\operatorname{Hom}}

\newcommand{\End}{\operatorname{End}}
\newcommand{\ext}{\operatorname{ext}}
\newcommand{\Ext}{\operatorname{Ext}}
\newcommand{\im}{\operatorname{Im}}
\newcommand{\codim}{\operatorname{codim}}
\newcommand{\coker}{\operatorname{coker}}

\newcommand{\xr}{\xrightarrow}

\newcommand{\ch}{\mathrm{ch}}

\newcommand{\stor}{\mathcal{T}or}
\newcommand{\Ku}{\operatorname{Ku}}
\newcommand{\sext}{\mathcal{E}xt}
\newcommand{\shom}{\mathcal{H}om}
\newcommand{\spanY}{V}
\newcommand{\At}{\mathrm{At}}
\newcommand{\tr}{\mathrm{Tr}}
\newcommand{\Hilb}{\mathrm{Hilb}}

\begin{document}
\bibliographystyle{amsalpha}

\sloppy


\title{Ulrich bundles on cubic fourfolds}
\author{Daniele Faenzi}
\address{Daniele Faenzi.
  Institut de Math{\'e}matiques de Bourgogne, UMR 5584 CNRS,
  Universit{\'e} de Bourgogne et Franche-Comt{\'e}, 9 Avenue Alain
  Savary, BP 47870, 21078 Dijon Cedex, France}
\email{daniele.faenzi@u-bourgogne.fr}

\author{Yeongrak Kim}
\address{Yeongrak Kim. Dept. of Mathematics, Pusan National University, 2 Busandaehak-ro 63beon-gil, Geumjeong-gu, 46241 Busan, Korea}
\email{yeongrak.kim@pusan.ac.kr} 

\thanks{D.F.
  partially supported by ISITE-BFC project contract ANR-lS-IDEX-OOOB and
 EIPHI Graduate School ANR-17-EURE-0002.
 Y.K. was supported by Project I.6 of the SFB-TRR 195
 ``Symbolic Tools in Mathematics and their Application'' of the German
 Research Foundation (DFG), Basic Science Research Program of the NRF of Korea (NRF-2021R1F1A1061140), and Pusan National University Research Grant, 2021. Both authors partially supported by
 F\'ed\'eration de Recherche Bourgogne Franche-Comt\'e Math\'ematiques (FR
 CNRS 2011).
}



\begin{abstract}
We show the existence of rank 6 Ulrich bundles on a smooth cubic
fourfold. First, we construct a simple sheaf $\mathcal E$ of rank 6 as
an elementary modification of an ACM bundle of rank 6 on a smooth
cubic fourfold. Such an $\mathcal E$ appears as an extension of two
Lehn-Lehn-Sorger-van Straten sheaves. Then we prove that a general
deformation of $\mathcal E(1)$ becomes Ulrich. In particular, this
says that  general cubic fourfolds have Ulrich complexity $6$.

\end{abstract}

\maketitle

\section*{Introduction}

An Ulrich sheaf on a closed subscheme $X$ of $\PP^N$
of dimension $n$ and degree $d$ is a non-zero coherent sheaf $\sF$ on $X$ satisfying $\rH^*(X,\sF(-j))=0$ for $1 \le j \le n$. 
In particular, the cohomology table $\{ h^i (X, \sF(j)) \}$ of $\sF$ is a multiple of the cohomology table of $\PP^n$. It turns out that the reduced Hilbert polynomial $\rp_\sF(t)=\chi(\sF(t))/\rank(\sF)$ of an Ulrich
sheaf $\sF$ must be:
\[
  \ru(t):=\frac{d}{n!}\prod_{i=1}^n(t+i).
\]

Ulrich sheaves first appeared in commutative algebra in the 1980s,
namely, in the form of maximally generated maximal Cohen-Macaulay modules
\cite{Ulr84}. Pioneering work of Eisenbud and Schreyer
\cite{ESW03} popularized them in algebraic geometry in
view of their many connections and applications. Eisenbud
and Schreyer asked whether every projective scheme supports an Ulrich
sheaf. That this should be the case is now called a conjecture of
Eisenbud-Schreyer (see also \cite{ES11}). They also proposed another
question about what is the smallest possible rank of an Ulrich sheaf on
$X$. This is called the \emph{Ulrich complexity} $\mathrm{uc}(X)$ of $X$
(cf. \cite{BES17}).  

Both the Ulrich existence problem and the Ulrich complexity problem
have been elucidated only for a few cases.
We focus here on the case when $X$ is a hypersurface in $\PP^{n+1}$
over an algebraically closed field 
$\kk$ of characteristic different from $2$.
Using the generalized Clifford algebra, Backelin and Herzog proved in
\cite{BH89} that any hypersurface $X$ has an Ulrich sheaf (even in
characteristic 2).
However, their
construction yields an Ulrich sheaf of rank
$d^{\tau(X)-1}$, where $\tau(X)$ is the Chow rank of
$X$ (i.e. the smallest length of an expression of the defining
equation of $X$ as sums of products of $d$ linear forms),
often much bigger than $\mathrm{uc}(X)$.

Looking in more detail at the Ulrich complexity problem for smooth
hypersurfaces of degree $d$ in $\PP^{n+1}$, the situation is well-understood for
arbitrary $n$ only for $d=2$.
Indeed, in this case the only indecomposable Ulrich bundles on $X$ are spinor
bundles, which have rank $2^{\lfloor (n-1)/2 \rfloor}$ \cite{BEH87}. On the
other hand, for $d \ge 3$, the Ulrich
complexity problem is wide open except for a very few
small-dimensional cases.
For instance, any smooth cubic curve or surface
$X$ satisfies $\mathrm{uc}(X)=1$, while for smooth cubic threefolds
$X$ we have $\mathrm{uc}(X)=2$,
(cf. \cite{Bea00, Bea02, LMS15}).

The main goal of this paper is to prove existence of Ulrich bundles $\sU$
defined on any smooth cubic fourfold $X$.
In particular the Chern character $\ch(\sU)$ should lie in $\QQ[H_X]$.
More precisely, 
$\ch(\sU(-1))=k \gamma$, with $\gamma = 3-H_X^2+\frac 14H^4_X$, for
some integer $k>1$. Our main result is:
\begin{mainthm} \label{main1}
  Any smooth cubic fourfold $X$ admits an Ulrich bundle $\sU$
  of rank $6$ with $\ch(\sU(-1))=2\gamma$.
  Hence there is 
  $M : 18\sO_{\PP^5}(-1) \to 18\sO_{\PP^5}$ with $\det(M)=f^6$,
  where $f$ is an equation of $X$.
\end{mainthm}

This allows to settle the Ulrich complexity problem for very general
cubic fourfolds. We know of no pair $(n,d)$ with $n \ge 5$, $d
\ge 4$ such that the Ulrich complexity problem of very general
hypersurfaces of degree $d$ in $\PP^n$ is solved; even for $n \in \{3,4\}$
the problem is open for large $d$.

\begin{cor*}
  If $X$ is a very general cubic fourfold, then $\mathrm{uc}(X)=6$.
\end{cor*}

To explain this, first note that
when $X$ is a smooth cubic fourfold then $X$ does not
support Ulrich bundles of rank $1$, but some $X$ can have an Ulrich
bundle $\sF$ of rank $2$, namely the pfaffian cubic
fourfolds; their moduli space forms a divisor
$\sC_{14} \subset \sC$, so a general cubic fourfold $X$ has
$\mathrm{uc}(X) \ge 3$.
A few more cubic fourfolds which have an Ulrich bundle of rank $3$ or $4$
have been reported very recently by Troung and Yen \cite{TY22}. However, all
these cases are special cubic fourfolds which contain a surface
 not homologous to a complete intersection. Indeed, it turns out that
 the Ulrich complexity of a very general cubic fourfold is divisible by $3$ 
and at least $6$, see \cite{KS20}. On the other hand, a
general cubic fourfold has a rank $9$ Ulrich bundle (cf. \cite{IM14,
  Man19, KS20}).

\medskip

Let us sketch briefly the strategy of the proof of Theorem \ref{main1}.
As a warm-up it, let us review a construction of a rank-$2$
Ulrich bundle on a smooth cubic threefold.
First, starting from a line $L$ contained in the threefold, one
constructs an ACM bundle of rank $2$ having $(c_1, c_2) = (0, L)$.
Such a bundle is unstable
since it has a unique global section which vanishes along $L$.
By choosing a line $L'$ disjoint from $L$, we may take an elementary
modification of it so that we have a simple and semistable sheaf $\sE$
of $(c_1, c_2) = (0,2L)$. The sheaf $\sE$ is not Ulrich, but one can show that its general deformation becomes Ulrich. A similar argument
is used to prove the existence of rank $2$ Ulrich bundles on K3
surfaces \cite{Fae19} and prime Fano threefolds \cite{BF11}.  

For fourfolds, twisted cubics play a central role in the
construction, rather than lines.
Note that twisted cubics in $X$ form a $10$-dimensional family. For each twisted
cubic $C \subset X$, its linear span $\spanY = \langle C \rangle $
defines a linear section $Y \subset X$ which is a cubic surface. When
$Y$ is smooth, the rank-$3$ sheaf $\mathcal G = \ker[3 \mathcal \sO_X
\to \sO_Y(C)]$ is stable.
The family of such stable sheaves of rank $3$ forms an $8$-dimensional
moduli space, which is indeed a very well studied 
smooth hyperk\"ahler manifold \cite{LLSS17, LLMS18}.
These sheaves have been used extensively in \cite{LLMS18,li-pertusi-zhao}.
We will call them Lehn-Lehn-Sorger-van Straten sheaves and the
Lehn-Lehn-Sorger-van Straten eightfold. 

To construct an Ulrich bundle of rank 6, we 
start from a twisted cubic $C \subset X$ and consider a rank 6
vector bundle $\sS$ obtained as fourth syzygy of $\sO_C(5)$. Then we take a
modification of $\sS$ along the cubic surface $Y$ obtained
cutting $X$ with the span of $C$, which one achieves
upon choosing a second twisted cubic $D$ in $Y$ that cuts
$C$ at 2 points.
This affords a sheaf $\sE$ which is certainly not reflexive but has
Chern character $2\gamma$ and enjoys 
almost all cohomology vanishing needed to be an Ulrich sheaf.

As it turns out, $\sE$ is a simple extension of the two
Lehn-Lehn-Sorger-van Straten sheaves of rank 3 associated with the two
twisted cubics $D$ and $C$ -- or rather its \textit{transpose} $C^\rt$, namely
the residual of $C$ in $2H_Y$.
The key point is that the sheaf $\sE$ lies
in the Kuznetsov category $\Ku(X)$ of $X$, \cite{Kuz04}.
This allows to obtain an Ulrich sheaf by taking a generic deformation
$\sF$ of $\sE$
 in the moduli space of simple sheaves over $X$ and using that the cohomology
 vanishing of $\sE$ propagates to $\sF$ by semicontinuity.
 This step relies on
deformation-obstruction theory of the sheaf as developed in \cite{KM09,BLMNPS19} and makes substantial use
of the fact that  $\Ku(X)$ is a K3 category.

Next, we argue about stability of our Ulrich bundles. It is well known
that Ulrich bundles are semistable. Here we prove that the ones we construct are generically stable.
Also, combining our bundles with the ones arising from \cite{IM14}, we
may construct stable Ulrich bundles of arbitrarily high rank.
This provides a higher-dimensional version of
the main results of \cite{CH11,CH12},
in the sense that $X$ is strictly Ulrich wild and thus verifies
\cite[Conjecture 1]{FPL17}. We do not know of other examples
of hypersurfaces of dimension $n \ge 4$ where this conjecture is known
to hold true.

\begin{mainthm} \label{main2}
  Given a smooth cubic fourfold, there is a $26$-dimensional symplectic family of
  stable Ulrich bundles of rank $6$. If $X$ is general enough, then for any
  $k > 1$ there is a $(6k^2+2)$-dimensional symplectic family of stable Ulrich
  bundles $\sU$ on $X$ with $\ch(\sU(-1))= k \gamma$.
\end{mainthm}

The above results can be thought of in terms of moduli of stable
objects of $\Ku(X)$ with respect to a Bridgeland stability condition
$\sigma$ on $\Ku(X)$.
Indeed, in the Mukai lattice of $\Ku(X)$, if we denote by $v_0$ the Mukai
vector of the object $\sG$ arising from a twisted cubic as above, then
the Mukai vector of our Ulrich sheaves is $2v_0$.
Then, our result implies that the moduli space of $\sigma$-stable
objects $\rM_\sigma(2v_0)$ has an irreducible component whose generic point is a
stable Ulrich bundle of rank $6$. 
It is likely that the spaces $\rM_\sigma(k v_0)$ are actually
irreducible for all $k > 1$. However, we are not aware of a proof of
this fact, nor of the answer to the following question.

\begin{question}
  Let $X$ be a smooth cubic fourfold, take $k >1$ and consider the
  Maruyama moduli space $\rM_X(k\gamma)$ of semistable sheaves of rank $k$
  with Chern character $k\gamma$.
  Is the open piece of $\rM_X(k \gamma)$ consisting of Ulrich bundles irreducible?
\end{question}

An intriguing question arises when looking at the space
$\rM_X(3\gamma)$, with $X$ very general. Indeed, the construction of
\cite{IM14} gives rise to Ulrich bundles of rank $9$ by realizing $X$
as a $\PP^5$-linear section of the Cartan cubic in $\PP^{26}$, which is equipped
with an $E_6$-equivariant Ulrich sheaf of rank $9$. 
The choice of the linear section is the equivalence class of a point
of $\GG(6,27)$ up the action of the 78-dimensional group $E_6$.
On a sufficiently general cubic fourfold $X$, this affords a
$28$-dimensional family of stable Ulrich bundles $\sU$ such that
$\sU(-1)$ lies in $\rM_X(3\gamma)$.
On the other hand, the family of stable Ulrich bundles $\sU$ with $\sU(-1)$ in $\rM_X(3\gamma)$ is 
symplectic of dimension $56$. We ask whether the Cartan bundles 
from \cite{IM14} form a Lagrangian subvariety of this family.
Perhaps this can be shown using the fact that rationally connected varieties admit no non-zero 2-form, but we have not been able to prove this rigorously.
Anyway, we do not know how to find a Lagrangian subvariety of $\rM_X(k
\gamma)$ for $k > 3$.

\medskip

The structure of this paper is as follows. In Section
\ref{section:background}, we recall basic notions and develop some
background mainly on Ulrich
bundles, syzygies and matrix factorizations.
In Section \ref{section:twisted cubics}, we introduce an ACM
bundle of rank $6$ which arises as a (higher) syzygy sheaf of a
twisted cubic and review some material on Lehn-Lehn-Sorger-van Straten sheaves as syzygy sheaves. Then
we take an elementary modification to define a strictly semistable
sheaf $\sE$ of rank $6$ whose reduced Hilbert polynomial is 
$\ru(t)$. In Section \ref{section:smoothing}, we show that a general
deformation of $\sE$ is Ulrich. We first prove this claim for cubic
fourfolds which do not contain surfaces of small degrees other than
linear sections. Then we extend it for every smooth cubic fourfold. 
Finally in Section \ref{k} we prove Theorem \ref{main2}.

\medskip
\begin{ack}
We wish to thank F\'ed\'eration Bourgogne Franche-Comt\'e Math\'ematiques FR CNRS 2011 for supporting the visit of Y.K. in Dijon.
We would like to thank Frank-Olaf Schreyer, Emanuele Macr\'i, Laura
Pertusi and Paolo Stellari, for valuable advice and helpful
discussion. We are also grateful to anonymous referees for many useful
remarks.
A part of this work was done while Y. K. was in Universit\"at des Saarlandes. 
We would like to thank the referee for the several suggestions on how
to improve the paper.
\end{ack}

\section{Background} \label{section:background}
Let us collect here some basic material. We work over an algebraically
closed field $\kk$ of characteristic other than $2$.

\subsection{Background definitions and notation}
Consider a smooth connected $n$-dimensional projective subvariety $X
\subseteq \PP^N$ and denote by $H_X$ the hyperplane divisor on
$X$ and $\sO_X(1) = \sO_X(H_X)$. Given a coherent sheaf $\sF$ on $X$
and $t \in \ZZ$, write $\sF(t)$ for $\sF \otimes \sO_X(t H_X)$.
Let $\sF$ be a torsion-free sheaf on $X$.
The reduced Hilbert polynomial of $\sF$ is defined as
\[
\rp_\sF(t) := \frac{1}{\rank(\sF)}\chi(\sF(t)) \in \QQ[t].
\]
Let $\sF, \sG$ be torsion-free sheaves on $X$. We say that $\rp_\sF < \rp_\sG$ if $\rp_\sF(t) < \rp_\sG(t)$ for $t \gg 0$. 
The slope of $\sF$
is defined as:
\[
  \mu(\sF)=\frac{c_1(\sF)\cdot H_X^{n-1}}{\rank(\sF)}.
\]
A torsion-free sheaf $\sF$ on $X$ is \emph{stable}
  (respectively, \emph{semistable, $\mu$-stable, $\mu$-semistable})
  if, for any subsheaf $0 \ne \sF' \subsetneq \sF$, we have:
\[
\rp_{\sF'} < \rp_{\sF}, \qquad \mbox{(respectively, $\rp_{\sF'} \le
  \rp_{\sF}$, $\mu(\sF') < \mu(\sF)$, $\mu(\sF') \le \mu(\sF)$} ).
\]
A polystable sheaf is a direct sum of stable sheaves having the same
reduced Hilbert polynomial.

\subsection{ACM and Ulrich sheaves}
We are mostly interested in coherent sheaves on $X$ which admit nice
minimal free resolutions over $\PP^N$, namely ACM and Ulrich
sheaves. Equivalently, such properties are characterized by cohomology
vanishing conditions as follows: 

\begin{defn}
Let $X \subseteq \PP^N$ be as above, and let $\sF$ be a coherent sheaf
on $X$. Then $\sF$ is:
\begin{enumerate}[label=\roman*)]
\item \emph{ACM} if it is locally Cohen-Macaulay and $\rH^i(X, \sF(j))=0$ for $0<i<n$ and $j \in \ZZ$.
\item \emph{Ulrich} if $\rH^i(X, \sF(-j))=0$ for $i \in \ZZ$ and $1 \le j \le n$. 
\end{enumerate}
\end{defn}
We refer to \cite[Proposition 2.1]{ESW03} for several equivalent
definitions for Ulrich sheaves. In particular, every Ulrich sheaf is
ACM. If $X$ is smooth, then a coherent sheaf is locally Cohen-Macaulay if and only if is locally free. Moreover, for Ulrich sheaves, 
  (semi-)stability is equivalent to  $\mu$-(semi-)stability, see \cite{CH12}.

Let us review some previous works on the existence of Ulrich bundles
on a smooth cubic fourfold $X$, possibly of small rank.
In terms of Hilbert polynomial, an Ulrich bundle $\sU$ satisfies:
\[
  \rp_\sU(t)=\ru(t) := \frac 18 (t+4)(t+3)(t+2)(t+1).
\]

Note that $X$ carries an Ulrich line bundle if and only if it is linearly determinantal, which is impossible since a determinantal hypersurface
is singular along a locus of codimension $3$. $X$ carries a rank $2$
Ulrich bundle if and only if it is linearly pfaffian. Equivalently,
such an $X$ contains a quintic del Pezzo surface \cite{Bea00}. Note that
a pfaffian cubic fourfold also carries a rank $5$ Ulrich
bundle \cite{Man19}. For rank $3$ and $4$, Truong and Yen provided computer-aided construction of a rank $3$ Ulrich bundle on a general element in the moduli of special cubic fourfolds $\mathcal C_{18}$ of discriminant
$18$, and of a rank $4$ Ulrich bundle on a general element in
$\mathcal C_{8}$ \cite{TY22}. 

All the above cases were made over special cubic fourfolds,
i.e., they contain a surface which is not homologous to a
complete intersection. Such cubic fourfolds form a countable union of
irreducible divisors in the moduli space of smooth cubic fourfolds
$\mathcal C$. We refer to \cite{Has00} for the convention and more
details. On a very general cubic fourfold $X$ (so that any surface
contained in $X$ is homologous to a complete intersection), it is easy
to find the following necessary condition on Chern classes of a
coherent sheaf to be Ulrich: 

\begin{prop}[{\cite[Proposition 2.5]{KS20}}]\label{Proposition:KS20}
Let $\sE$ be an Ulrich bundle of rank $r$ on a very general cubic fourfold $X \subset \mathbb{P}^5$. Let $c_i := c_i (\sE(-1))$. Then $r$ is divisible by $3$, $r \ge 6$, and 
\[
c_1 = 0, \qquad c_2 = \frac{1}{3}r H^2, \qquad c_3 = 0, \qquad c_4 = \frac{1}{6}r(r-9).
\]
\end{prop}

The existence of rank $9$ Ulrich bundles on a general cubic fourfold
$X$ is known according to \cite{IM14,
  Man19, KS20}. Therefore, the Ulrich complexity of a very general
cubic fourfold is either $6$ or $9$.
It is thus natural to ask the question: \textit{Does a smooth cubic fourfold carry an Ulrich bundle of rank $6$}? The goal of this paper is to give a
positive answer to this question.
In particular, the Ulrich complexity $\mathrm{uc}(X)$ of a (very)
general cubic fourfold $X$ is $6$.

\subsection{Reflexive sheaves}
Let $\sE$ be a torsion-free sheaf on a smooth connected projective $n$-dimensional
variety $X$.  The following lemma is standard.

\begin{lem} \label{reflexive}
For each $k \in \{0,\ldots,n-2\}$ there is $p_k \in \QQ[t]$, with
  $\deg(p_k) \le k$ such that:
  \[
    \rh^{k+1}(\sE(-t))=p_k(t), \qquad \rh^0(\sE(-t))=0\qquad  \mbox{for $t \gg 0$}, 
  \]
  Assume $\sE$ is reflexive. Then $\forall k \in \{0,\ldots,n-3\}$ there is $q_k \in \QQ[t]$, with
  $\deg(q_k) \le k$ such that:
  \[
    \rh^{k+2}(\sE(-t))=q_k(t), \qquad \rh^0(\sE(-t))=\rh^1(\sE(-t))=0 \qquad  \mbox{for $t \gg 0$}, 
  \]
  Moreover, $\sE$ is locally free if and only if $p_k=0$ for all $k
  \in \{0,\ldots,n-2\}$, equivalently if $q_k=0$ for all $k  \in \{0,\ldots,n-3\}$.
\end{lem}

\begin{proof}
  Given positive integers $p,q$ with $p+q \le n$, Serre duality and the
  local-global spectral sequence give, for all $t \in \ZZ$:
  \begin{equation}
    \label{SS}
    \rH^{n-p-q}(\sE(-t))^\vee \simeq \Ext^{p+q}_X(\sE,\omega_X(t))
    \Leftarrow \rH^p(\sext_X^q(\sE,\omega_X) \otimes \sO_X(t)) = E_2^{p,q}.
  \end{equation}

  For $t \gg 0$ and $p > 0$ we have $\rH^p(\sext_X^q(\sE,\omega_X)
  \otimes \sO_X(t))=0$ by Serre vanishing. Then:
  \[
    \rh^{n-q}(\sE(-t)) = \rh^0(\sext_X^q(\sE,\omega_X) \otimes
    \sO_X(t)), \qquad \mbox{for $t \gg 0$}.
  \]
  Hence $\rh^{n-q}(\sE(-t))$ is a rational polynomial function of
  $t$ for $t \gg 0$.
  By \cite[Proposition 1.1.10]{HL10}, since $\sE$ is torsion-free, for $q \ge 1$ we have:
  \[
    \codim(\sext_X^q(\sE,\omega_X)) \ge q+1,
  \]
  while when $\sE$ is reflexive, for $q \ge 1$:
  \[
    \codim(\sext_X^q(\sE,\omega_X)) \ge q+2.
  \]
  Thus, for $t \gg 0$, the degree of the polynomial function
  $\rh^{n-q}(\sE(-t))$ is at most $n-q-1$, actually at most $n-q-2$ if
  $\sE$ is reflexive.

  Finally, $\sE$ is locally free if and only if
  $\sext_X^q(\sE,\omega_X)=0$ for all $q>1$. Since this happens if and
  only if ${\rh^0(\sext_X^q(\sE,\omega_X) \otimes \sO_X(t)) =0}$ for $t \gg
  0$, the last statement follows.
\end{proof}

\subsection{Minimal resolutions and syzygies} 
We recall some notions from commutative algebra.
Let $R = \kk [x_0,\cdots, x_N]$ be a polynomial ring over a field $\kk$ with the
standard grading, and let $R_X = R/I_X$ be the homogeneous coordinate
ring of $X$ where $I_X$ is the ideal of $X$. Let $\Gamma$ be a
finitely generated graded $R_X$-module.
The minimal free resolution of $\Gamma$ over $R_X$ is constructed by
choosing minimal generators of $\Gamma$ of degrees
$(a_{0,0},\ldots,a_{0,r_0})$ so that there is a surjection 
\[
F_0 = \bigoplus_{j=0}^{r_0} R_X (-a_{0,j}) \epi \Gamma.
\]
Taking a minimal set of generators of degrees
$(a_{1,0},\ldots,a_{1,r_1})$ of its kernel we get a minimal
presentation of $\Gamma$ of the form $F_1 = \bigoplus_{j=0}^{r_1}
R_X(-a_{1,j}) \to F_0$. Repeating this process, we have a free resolution of $\Gamma$:
\[
F_{\bullet} (\Gamma) : \cdots \to F_i \stackrel{d_i} \longrightarrow
F_{i-1} \stackrel{d_{i-1}} \longrightarrow \cdots \longrightarrow F_1
\stackrel{d_1} \longrightarrow F_0 \to \Gamma \to 0, \qquad
\mbox{with:} \qquad F_i = \bigoplus_{j=0}^{r_i} R_X (-a_{i,j}).
\]
Note that the resolution obtained this way is minimal, i.e. $d_i
\otimes_{R_X} \kk = 0$ for every $i$, and is unique up to homotopy,
see \cite[Corollary 1.4]{Eis80}. In general, it has infinitely many terms.

We define the \textit{minimal resolution} of a coherent sheaf $\sF$ on
$X$ as the sheafification of the minimal graded free resolution of its
module of global sections $\Gamma_{*}(\sF) = \bigoplus_{j \in \ZZ} \Gamma(X, \sF (j))$, provided that this is finitely
generated. In this case, for $i \in \NN$, we call \textit{i-th syzygy}
of $\sF$ the sheafification of $\im(d_i)$ and we denote this by
$\Sigma^X_i(\sF)$. Of course for positive $j$ we have $\Sigma^X_{i+j}(\sF) \simeq \Sigma^X_j\Sigma^X_i(\sF)$.

\subsection{Matrix factorizations and ACM/Ulrich sheaves}
We recall the notion of matrix factorization which is introduced by Eisenbud \cite{Eis80} to study free resolutions over hypersurfaces. 

\begin{defn}
Let $X \subseteq \PP^N$ be a hypersurface defined by a homogeneous
polynomial $f$ of degree $d$, and let $\sF$ and $\sG$ be two finite direct sums
of line bundles. A pair of morphisms $\varphi : \sF \to \sG$ and $\psi : \sG (-d) \to \sF$ is called a \emph{matrix factorization} of $f$ (of $X$) if
\[
\varphi \circ \psi = f \cdot id_{\sG(-d)}, \qquad \psi (d) \circ \varphi = f \cdot id_{\sF}.
\]
\end{defn}

Matrix factorizations have a powerful application to ACM/Ulrich bundles as follows:

\begin{prop}[{\cite[Corollary 6.3]{Eis80}}]
The association 
\[
(\varphi, \psi) \mapsto M_{(\varphi, \psi)} := \coker \varphi
\]
induces a bijection between the set of equivalence classes of reduced matrix factorizations of $f$ and the set of isomorphism classes of indecomposable ACM sheaves. In particular, when $(\varphi, \psi)$ is completely linear, that is, $\varphi : \sO_{\PP^N}(-1)^{\oplus t} \to \sO_{\PP^N}^{\oplus t}$ for some $t \in \ZZ$ then the corresponding sheaf is Ulrich.
\end{prop}

\subsection{Twisted cubics and Lehn-Lehn-Sorger-van Straten eightfold}\ Let us briefly recall how can we construct a rank 2 Ulrich bundle on a cubic threefold $X$ via deformation theory. If there is such an Ulrich bundle $\sF$, then $\sF(-1)$ must have the Chern classes $(c_1, c_2) = (0, 2)$ by Riemann-Roch. Note that $X$ has an ACM bundle $\sF_1$ of rank $2$ with $(c_1, c_2) = (0,1)$ which fits into the following short exact sequence
\[
0 \to \sO_X \to \sF_1 \to \sI_{\ell} \to 0
\]
where $\ell \subset X$ is a line. We see that $\sF_1$ is unstable due
to its unique global section. We can take an elementary modification
with respect to $\sO_{\ell^{\prime}}$ where $\ell^{\prime} \subset X$
is a line disjoint to $\ell$. The resulting sheaf $\sF_2 := \ker
\left[ \sF_1 \to \sO_{\ell^{\prime}} \right]$ is simple, strictly
semistable, and non-reflexive. One can check that its general
deformation is stable and locally free, and becomes Ulrich after
twisting by $\sO_X(1)$. One major difference between the case of cubic
threefolds and fourfolds is that not lines but twisted cubics play a significant role both in finding an ACM bundle (of same $c_1$ as Ulrich) and taking an elementary modification. 

Let $X \subset \PP^5$ be a smooth cubic fourfold which does not contain a plane, and let $M_3 (X)$ be the irreducible component of the Hilbert scheme of $X$ containing the twisted cubics. Then $M_3(X)$ is a smooth irreducible projective variety of dimension $10$ \cite[Theorem A]{LLSS17}. Let $C$ be a twisted cubic contained in $X$, and $V \simeq \PP^3$ be its linear span. According to \cite{LLSS17}, the natural morphism $C \mapsto V \in Gr(4,6)$ factors through a smooth projective eightfold $\sZ^{\prime}$ so that $M_3(X) \to \sZ^{\prime}$ is a $\PP^2$-fibration. In $\sZ^{\prime}$, there is an effective divisor coming from non-CM twisted cubics on $X$ which induces a further contraction $\sZ^{\prime} \to \sZ$ so that $\sZ$ is a smooth hyperk\"ahler eightfold which contains $X$ as a Lagrangian submanifold, and the map $\sZ^{\prime} \to \sZ$ is the blow-up along $X$ \cite[Theorem B]{LLSS17}. The variety $\sZ$ is called the Lehn-Lehn-Sorger-van Straten eightfold.

We are interested in a moduli description of
$\sZ^{\prime}$. Let $Y := V \cap X$ be a cubic surface containing
$C$. The sheaf $\sI_{C/Y} (2)$ is indeed an Ulrich line bundle on $Y$, and hence it fits into the following short exact sequence 
\[
0 \to \sG_C \to 3 \sO_X \to \sI_{C/Y}(2) \to 0.
\]
Lahoz, Lehn, Macr\`i and Stellari showed that the sheaf $\sG_C$ is
stable, and the moduli space of Gieseker stable sheaves with the same Chern character is isomorphic to $\sZ^{\prime}$ \cite{LLMS18}. Since we are only interested in general CM twisted cubics and corresponding Lehn-Lehn-Sorger-van Straten sheaves, we may regard that a general point of the Lehn-Lehn-Sorger-van Straten eightfold $\sZ$ corresponds to a rank 3 sheaf $\sG_C$, where $C$ is a CM twisted cubic on $X$, even when $X$ potentially contains a plane.

\section{Syzygies of twisted cubics} \label{section:twisted cubics}

Let $X \subset \PP^5$ be a smooth cubic fourfold. 

\subsection{Twisted cubics and $6$-bundles}
\label{Subsection:construction of S}

Here we show that taking the fourth syzygy of the structure sheaf of
a twisted cubic $C$ is a vector
bundle of rank $6$ which admits a trivial subbundle of rank
$3$. Factoring out this quotient gives back the second syzygy of
$C$, with a degree shift. We will use this filtration later on.

\begin{prop} \label{syzygy of twisted}
  Let $C \subset X$ be a twisted cubic, $\spanY$ its linear span and set $Y=X\cap \spanY$. Put:
  \[
    \sS=\Sigma^X_4(\sO_C(5)), \qquad
     \sG_C=\Sigma^X_1(\sI_{C/Y}(2)).
  \]
  Then $\sS$ is an ACM sheaf of rank $6$ on $X$ with:
  \[
    \rp_\sS(t)=\frac 1 8 (t+2)^2(t+1)^2, \qquad \rH^*(\sS(-1))=\rH^*(\sS(-2))=0.
  \]
  Moreover, $\rh^0(X,\sS)=3$ and there is an exact sequence:
  \begin{equation}
    \label{hatSO}
    0 \to 3\sO_X \to \sS \to \sG_C \to 0.    
  \end{equation}
  In particular, we have:
  \[
  { \ch(\sG_C) = \gamma = 3-H_X^2 + \frac{1}{4} H_X^4, \qquad \ch(\sS)= 6 - H_X^2 + \frac{1}{4} H_X^4.}
  \]
\end{prop}

To keep notation lighter, we remove the subscript $C$ from $\sG_C$ so we just write $\sG$, as soon as no confusion occurs, i.e. until \S \ref{second conic}.

\begin{proof}
  For the sake of this proof, for any integer $i$ we omit writing
  $\sO_C$ from expressions of the form $\Sigma_i^X(\sO_C)$ and
  $\Sigma_i^Y(\sO_C)$, so that for instance:
  \begin{equation}
    \label{SigmaI}
    \Sigma_1^X \simeq \sI_{C/X}, \qquad
    \Sigma_1^Y \simeq \sI_{C/Y}.
  \end{equation}

  The curve $C$ is Cohen-Macaulay of degree $3$ and arithmetic genus
$0$, its linear span $\spanY$ is a $\PP^3$, and the linear section $Y$ is a cubic surface equipped with the Ulrich line bundle $\sI_{C/Y}(2)$.
Hence, we have a linear resolution on $\spanY$:
\begin{equation}
  \label{eq:M}
    0 \to 3\sO_\spanY(-3) \xr{M} 3\sO_\spanY(-2) \to \sI_{C/Y} \to 0,  
\end{equation}
  where $M$ is a matrix of linear forms whose determinant is an
  equation of $Y$ in $\spanY$. Put $G_1=3\sO_Y(-2)$ and
  $G_2=3\sO_Y(-3)$.
  Thanks to \cite[Theorem 6.1]{Eis80},
  taking the adjugate matrix $M'$ of $M$ forms a matrix factorization $(M, M')$
  of $Y$ which provides the following $2$-periodic resolution on $Y$
  (we still denote by $M, M'$  the reduction of $M$ and $M'$ modulo $Y$):
  \begin{equation}
    \label{infiniteY}
    \cdots \xr{M'} G_2(-3) \xr{M} G_1(-3) \xr{M'} G_2 \xr{M} G_1 \to \sI_{C/Y} \to 0.
  \end{equation}
  This gives, for all $i \in \NN$:
  \begin{equation}
    \label{periodic}
    \Sigma^Y_{2i+1} \simeq \sI_{C/Y}(-3i).
  \end{equation}
  Next,  set $K_0 = \sO_X$, $K_1=2\sO_X(-1)$, $K_2 = \sO_X(-2)$ and write the Kozsul resolution:
  \begin{equation}
    \label{koszulY}
    0 \to K_2 \to K_1 \to K_0 \to \sO_Y \to 0.
  \end{equation}
  Now we look at the exact sequence:
  \begin{equation}
    \label{ideals}
    0 \to \sI_{Y/X} \to \sI_{C/X} \to \sI_{C/Y} \to 0.    
  \end{equation}

  Set $F_1=3\sO_X (-2)$, $F_2=3\sO_X(-3)$. 
  Now we proceed in two directions. On one hand, the composition
  $F_1 \to G_1 \to \Sigma_1^Y$ lifts to $F_1 \to \Sigma_1^X$ to give
  a diagram (we
  omit zeroes all around for brevity):
  \[
    \xymatrix@-2ex{
      K_2 \ar[d]  \ar[r] & K_1 \ar[d] \ar[r] & \sI_{Y/X} \ar[d]\\
      \Sigma_2^X \ar[d] \ar[r] & F_1 \oplus K_1 \ar[d] \ar[r] & \Sigma_1^X \ar[d] \\
      \Sigma_1^X\Sigma_1^Y\ar[r] &F_1 \ar[r] & \Sigma_1^Y
    }
  \]

{The leftmost column of the above diagram says that
  $\Sigma_1^X\Sigma_1^Y$ and $\Sigma_2^X$ agree up to the free factor
  $K_2$ and therefore their higher syzygies are the same.
  More precisely we have $\Gamma_*(K_2) \simeq R_X(-2)$, so
  the sheafified minimal resolutions of $\Sigma_2^X$
  and $\Sigma_1^X\Sigma_1^Y$ over $X$ differ only by the term $\sO_X(-2)$ in
  degree 0, so that
  $\Sigma_3^X \simeq \Sigma_2^X\Sigma_1^Y$ and in turn
  $\Sigma_{i+1}^X \simeq \Sigma_i^X \Sigma_1^Y$ for all $i \ge 2$.
  Summing up we have:
  \begin{equation}
    \label{greater3}
    0 \to K_2 \to \Sigma_2^X \to \Sigma_1^X\Sigma_1^Y \to 0, \qquad
    \Sigma_{i+1}^X \simeq \Sigma_i^X\Sigma_1^Y, \qquad \forall i \ge 2.
  \end{equation}}

  Next, \eqref{ideals},
  \eqref{SigmaI} and \eqref{infiniteY}
induce a diagram 
  \[
    \xymatrix@-2ex{
      F_1 \otimes \sI_{Y/X} \ar@{=}[d]\ar[r] & \Sigma_1^X \Sigma_1^Y
      \ar[r] \ar[d] &
      \Sigma_2^Y\ar[d] \\
      F_1 \otimes \sI_{Y/X} \ar[r] & F_1 \ar[r] \ar[d] & G_1 \ar[d]\\
      & \Sigma_1^Y \ar@{=}[r] & \Sigma_1^Y
    }
  \]
  This in turn gives the exact sequence
  \begin{equation}
    \label{hatS}
    0 \to F_1 \otimes \sI_{Y/X} \to \Sigma_1^X\Sigma_1^Y \to   \Sigma_2^Y \to 0.   \end{equation}

  Lifting $F_2 \to \Sigma_2^Y$ to $F_2 \to \Sigma_1^X\Sigma_1^Y$,
  we get the exact diagram:
  \[
    \xymatrix@-2ex{
      F_1 \otimes K_2 \ar[r] \ar[d]  &F_1 \otimes K_1 \ar[d] \ar[r] & F_1 \otimes \sI_{Y/X} \ar[d] \\
      \Sigma_2^X\Sigma_1^Y  \ar[d] \ar[r] & F_1 \otimes K_1 \oplus F_2
      \ar[r] \ar[d] & \Sigma_1^X\Sigma_1^Y \ar[d] \\
      \Sigma_1^X\Sigma_2^Y \ar[r] & F_2 \ar[r]&   \Sigma_2^Y.
    }
  \]

  Using the diagram and the fact that $\Gamma_*(F_1 \otimes K_2)$ is
  free we get:
  \begin{equation}
    \label{greater2}
      0 \to F_1 \otimes K_2 \to \Sigma_2^X\Sigma_1^Y \to  \Sigma_1^X
    \Sigma_2^Y \to 0, \qquad     \Sigma_{i+1}^X\Sigma_1^Y \simeq
    \Sigma_i^X\Sigma_2^Y, \forall i \ge 2.
  \end{equation}

  Repeating once more this procedure and using the periodicity of \eqref{infiniteY} we get:
  \[
    0 \to F_2 \otimes \sI_{Y/X} \to \Sigma_1^X \Sigma_2^Y \to \Sigma_3^Y \to 0.
  \]
  Then, using \eqref{periodic} and lifting $F_1(-3) \to \sI_{C/Y}(-3) \simeq \Sigma_3^Y$ to $F_1(-3) \to \Sigma_1^X \Sigma_2^Y$, we have the exact sequence:
  \[
    0 \to F_2 \otimes K_2 \to \Sigma_2^X \Sigma_2^Y \to \Sigma_1^X
    \Sigma_3^Y \to 0.
  \]

  Summing up, \eqref{greater3} and \eqref{greater2} give $\Sigma_4^X \simeq
  \Sigma_3^X\Sigma_1^Y \simeq \Sigma_2^X\Sigma_2^Y$, so that the above
  sequence tensored with $\sO_X(5)$ becomes:
  \[
    0 \to 3\sO_X \to \sS \to \Sigma_1^X(\sI_{C/Y}(2)) \to 0,
  \]
  which is the sequence appearing in the statement. The fact that
  $\rh^0(X,\sS)=3$ is clear from the sequence.
  Since $X$ is smooth and $C
  \subset X$ is arithmetically Cohen-Macaulay of
  codimension $3$, the syzygy sheaf $\Sigma_4^X$ is ACM and hence locally free.
  Looking at the above resolution we compute the following invariants of
  $\sS$:
  \[
    \rank(\sS)=6, \qquad c_1(\sS) = 0, \qquad c_2(\sS)=H^2, \qquad
    \rp_\sS(t)=\frac 1 8 (t+1)^2(t+2)^2.
  \]

  It remains to prove $\rH^*(\sS(-1))=\rH^*(\sS(-2))=0$. By
  \eqref{hatSO}, it suffices to show $\rH^*(\sG(-1))=\rH^*(\sG(-2))=0$. By definition we have
   \begin{equation}
     \label{SI}
     0 \to \sG \to 3\sO_X \to \sI_{C/Y}(2) \to 0,
   \end{equation}
   and $\sI_{C/Y}(2)$ is Ulrich on $Y$ so
   $\rH^*(\sI_{C/Y}(1))=\rH^*(\sI_{C/Y})=0$. We conclude that
   $\rH^*(\sG(-1))=\rH^*(\sG(-2))=0$. 
{The Chern characters for $\sG_C$ and $\sS$ can be computed
  immediately from \eqref{SI} and \eqref{hatSO}, see also \cite[Section 2.2]{LLMS18}}. 
\end{proof}

Along the way we found the following minimal free resolution of
$\sO_C$ over $X$:
  \[
    \begin{array}{ccccccccc}
      &  3 \sO_X(-5) & & 9\sO_X(-4) & & \sO_X(-2) & & 2\sO_X(-1) &  \\
      \cdots \to &  \oplus & \stackrel{d_4} \to & \oplus &\stackrel{d_3} \to & \oplus & \stackrel{d_2} \to & \oplus & \stackrel{d_1} \to  \sO_X \to \sO_C \to 0.\\
      & 9\sO_X(-6) & & 3\sO_X(-5) & & 9 \sO_X(-3)& & 3\sO_X(-2) & 
    \end{array}
  \]
  This is an instance of Shamash's resolution. It becomes periodic after three steps.
  We record that $\sS$ fits into:
\[
\xymatrixrowsep{10pt}
\xymatrixcolsep{10pt}
\xymatrix{
\cdots \ar[r] & 9 \sO_X(-2) \oplus 3 \sO_X(-3) \ar^-{d_5}[rr] \ar[dr]  && 3 \sO_X \oplus 9 \sO_X(-1)
\ar[dr] \ar^-{d_4}[rr]&& 9 \sO_X(1) \oplus 3 \sO_X \ar[r] & \cdots \\
& & \Sigma_5^X(\sO_C(5)) \ar[ur] & & \sS \ar[ur]  & & &
}
\]
 Recall that the matrix $M$ of linear forms presents
  $\sI_{C/Y}$.
  We see it as a map of $\sO_Y$-modules with target in 
  $3\sO_Y$ and denote its image by  $\sR$:
  \begin{equation}
    \label{defR}
\sR = \Sigma_1^Y(\sI_{C/Y}(2)) \simeq \im(M), \qquad \mbox{with}
\qquad M : 3\sO_Y(-1) \to 3\sO_Y.    
  \end{equation}
{From  \eqref{infiniteY} we get the exact sequences:
    \begin{equation}
      \label{morexp}
      0 \to \sI_{C/Y} (-1) \to 3 \sO_Y (-1) \to \sR \to 0, \qquad 0 \to \sR \to
      3 \sO_Y \to \sI_{C/Y}(2) \to 0.
    \end{equation}}

The following lemma is essentially \cite[Proposition 2.5]{LLMS18}, we
reproduce it here for self-containedness. In fact, given a Cohen-Macaulay
twisted cubic $C \subset X$, the sheaf $\sG=\sG_C$
represents uniquely a point of the Lehn-Lehn-Sorger-van Straten
eightfold $\sZ$ associated with the cubic fourfold $X$.

\begin{lem} \label{hat S is stable}
  Assume that $Y$ is integral. Then the sheaf $\sG$ is stable with:
  \[
    \rp_{\sG}(t)=\ru(t-1)=\frac 18 (t+3)(t+2)(t+1)t, \qquad  \rH^*(X,\sG(-t))=0, \qquad \mbox{for $t=0,1,2$}.
  \]
  Finally, we have $\sext_X^i(\sG,\sO_X)=0$ except for $i=0,1$,
  in which case:
  \[\sG^\vee \simeq 3\sO_X,
    \qquad \sext_X^1(\sG,\sO_X) \simeq \shom_Y(\sI_{C/Y},\sO_Y) = \sO_Y(C).
    \]
\end{lem}

\begin{proof}
  The Hilbert polynomial of $\sG$ is computed directly from the
  previous proposition. Next, we use the sheaf $\sR$ which satisfies
  $\sR \simeq \Sigma_2^Y(\sO_{C}(2))$.
  Recall from the proof of the previous proposition the sequence
  \eqref{hatS} that we rewrite as:
  \begin{equation}
    \label{IR}
     0 \to 3 \sI_{Y/X} \to \sG \to \sR \to 0.    
  \end{equation}
    
  By definition of $\sG
   = \Sigma_1^X(\sI_{C/Y}(2))$, the map $3\sO_X \to \sI_{C/Y}(2)$ in \eqref{SI} 
   induces an isomorphism on global sections, hence  
   $\rH^*(X,\sG)=0$. The vanishing $\rH^*(X,\sG(-1))=\rH^*(X,\sG(-2))=0$ was proved in the previous proposition.

   Next, we show first that $\sG$ is simple. 
   Applying $\Hom_X(-,\sG)$ to \eqref{SI}, we get:
   \[
     \End_X(\sG) \simeq \Ext_X^1(\sI_{C/Y}(2),\sG).
   \]
   We note that $\sI_{C/Y}$ is simple,
   $\Hom_X(\sI_{C/Y}(2),\sO_X)=0$ as $\sI_{C/Y}$ is torsion
   and:
   \[
     \Ext_X^1(\sI_{C/Y}(2),\sO_X) \simeq \rH^3(\sI_{C/Y}(-1))^\vee=0
   \]
   since $\dim(Y)=2$. Hence applying $\Hom_X(\sI_{C/Y}(2),-)$ to 
   \eqref{SI}, we observe that $\sG$ is simple:
   \[
     \End_X(\sG) \simeq  \Ext_X^1(\sI_{C/Y}(2),\sG) \simeq \End_X(\sI_{C/Y}) \simeq \kk.
   \]

   Suppose that $\sG$ is not stable. Consider a saturated destabilizing subsheaf $\sK$ of $\sG$ so $\rank(\sK)
   \in \{1,2\}$ and $\rp_\sK \ge \rp_{\sG}$ so that $\sQ = \sG / \sK$ is torsion-free with $\rank(\sQ)=3-\rank(\sK)$. Since
   $\sK \subset \sG \subset 3\sO_X$, we have $\mu(\sK) \le
   0$. From $\rp_\sK \ge \rp_{\sG}$ we deduce that
   $c_1(\sK)=c_1(\sQ)=0$.

   We look at the two possibilities for $\rank(\sK)$.
   If $\rank(\sK)=1$, then $\sK$ is torsion-free with $c_1(\sK)=0$
   so there is a closed subscheme $Z \subset X$  of codimension at
   least $2$ such that $\sK \simeq
   \sI_{Z/X}$. If $Z=\emptyset$ then $\sK \simeq \sO_X$, which is
   impossible as $\rH^0(X,\sG)=0$. Now, for $Z \ne \emptyset$,
   consider the inclusion $\sI_{Z/X} \subset \sG
   \subset 3\sO_X$. Taking reflexive hulls, we see that this factors
   through a single copy of $\sO_X$ in $3\sO_X$. Looking at
   \eqref{SI}, we get that the quotient
   $\sO_Z=\sO_X/\sI_{Z/X}$ inherits a non-zero map to
   $\sI_{C/Y}(2)$. The image of this map is $\sO_Y$ itself
   because $\sI_{C/Y}(2)$ is torsion-free of rank $1$ over $Y$ as $Y$ is integral.

   Note that $\rp_{\sI_{Z/X}} = \rp_{\sG}$ precisely when 
   $Z$ is a linear subspace $\PP^2$ contained in $X$, and that
   $\rp_{\sI_{Z/X}} < \rp_{\sG}$ if $\deg(Z) \ge 2$ and
   $\dim(Z)=2$.
   Hence, the image of $\sO_Z \to \sI_{C/Y}(2)$ cannot be the whole $\sO_Y$ as
   then $Y \subseteq Z$, so we have $\dim(Z) = 2$ and $\deg(Z) \ge 3$, while we
   are assuming $\rp_{\sI_{Z/X}} \ge \rp_{\sG}$.  
   Therefore, the possibility $\rank(\sK)=1$ is ruled out.
   \smallskip
   
   Now we may assume $\rank(\sK)=2$. Arguing as in the previous case,
   we deduce that there is a closed
   subscheme $Z \subset X$ of codimension at least $2$ such that $\sQ \simeq \sI_{Z/X}$.
   Using \eqref{IR} and noting that $3\sI_{Y/X}$ cannot be contained in $\sK$
   for $\rank(\sK)=2$, we get a non-zero map $3\sI_{Y/X} \to
   \sI_{Z/X}$ by composing $3 \sI_{Y/X} \mono \sG$ with $\sG
   \epi \sI_{Z/X}$. The image of this map is of the form $\sI_{Z'/X}
   \subset \sI_{Z/X}$ for some closed subscheme $Z' \supseteq Z$ of $X$. Since $3\sI_{Y/X}$ is polystable and
   $3\sI_{Y/X} \epi \sI_{Z'/X}$, we have $\sI_{Z'/X} \simeq \sI_{Y/X}$
   so $Z'=Y$. In particular, we have $Z \subseteq Y$.

   Again, we use that $\rp_{\sI_{Z/X}} > \rp_{\sG}$ as soon as
   $\dim(Z) \le 1$, so the
   assumption that $\sK$ destabilizes $\sG$ forces 
   $\dim(Z) \ge 2$. Hence, $Z$ is a surface contained in $Y$ so that $Z=Y$
   since $Y$ is integral. Then $\sI_{Y/X}$ is a direct summand of $\sG$ which therefore splits as $\sG = \sK \oplus \sI_{Y/X}$. But
   this contradicts the fact that $\sG$ is simple.
   We conclude that $\sG$ must be stable.

   Finally, we apply $\shom_X(-,\omega_X)$ to \eqref{SI} and use
   Grothendieck duality to compute $\sext^i_X(\sG,\sO_X)$ using
   that $\sI_{C/Y}$ is reflexive on $Y$ to get:
   \[\sext^1_X(\sG,\omega_X) \simeq \sext^2_X(\sI_{C/Y} (2),\omega_X) \simeq
     \shom_Y(\sI_{C/Y} (2),\omega_Y).
   \]
   Since $\omega_X \simeq \sO_X(-3)$ and $\omega_Y \simeq \sO_Y(-1)$,
   the conclusion follows.   
 \end{proof}

 The next lemma analyzes the restriction of $\sS$ onto $Y$.

 \begin{lem} \label{SonY}
   There is a surjection $\xi : \sS|_Y \to \sR$ whose kernel fits into:
   \begin{equation}
     \label{kerxi}
     0 \to \sR(1) \to \ker(\xi) \to 2 \sI_{C/Y}(1) \to 0.     
   \end{equation}
 \end{lem}

 \begin{proof}
   First of all, restricting the Koszul resolution \eqref{koszulY} to
   $Y$ we find:
   \[
     \stor_1^X(\sI_{C/Y},\sO_Y) \simeq 2 \sI_{C/Y}(-1), \qquad
     \stor_2^X(\sI_{C/Y},\sO_Y) \simeq \sI_{C/Y}(-2).
   \]
   Therefore, restricting \eqref{SI} to $Y$ we get:
   \begin{equation}
     \label{4t}
     0 \to 2 \sI_{C/Y}(1) \to \sG|_Y \to 3\sO_Y \to \sI_{C/Y}(2)
     \to 0,
   \end{equation}
   and hence:
   \begin{equation}
     \label{hatS|Y}
     0 \to 2 \sI_{C/Y}(1) \to \sG|_Y \to \sR \to 0.     
   \end{equation}
   We also get:
   \[
     \stor_1^X(\sG,\sO_Y) \simeq \stor_2^X(\sI_{C/Y}(2),\sO_Y) \simeq \sI_{C/Y}.
   \]
   {Therefore, by the previous display, restricting \eqref{hatSO} to
     $Y$ we obtain the exact sequence:
   \[
     0 \to \sI_{C/Y} \to 3\sO_Y \to \sS|_Y \to \sG|_Y \to 0.
   \]
   Then, comparing \eqref{morexp} and the leftmost part of the above
   sequence, we see that the image of the middle
   map is $\sR(1)$.
 Hence we extract from \eqref{4t} the exact sequence:
   \begin{equation}
     \label{sS|Y}
     0 \to \sR(1) \to \sS|_Y \to \sG|_Y \to 0.     
   \end{equation}}
   Composing $\sS|_Y \to \sG|_Y$ with the surjection appearing in
   \eqref{hatS|Y} we get the surjection $\xi$. Using \eqref{hatS|Y}
   and \eqref{sS|Y} we get the desired filtration for $\ker(\xi)$.
 \end{proof}

\subsection{Elementary modification along a cubic surface}
\label{second conic}

In \S \ref{Subsection:construction of S} we constructed an ACM bundle
$\sS$ of rank $6$. Recall that $\rh^0(X,\sS)=3$, and these three global
sections of $\sS$ make it unstable. Hence, it is natural to consider
an elementary modification of $\sS$ by a sheaf $\sA$ such that  
$H^0(\sS) \stackrel{\sim} \to H^0(\sA)$. Moreover, Proposition
\ref{Proposition:KS20} suggests a good candidate for $\sA$ to get
closer to an Ulrich bundle on $X$. Indeed, we should have:
\[
  \chi_\sA(t)=6\rp_\sS(t)-6\ru(t-1)= \frac 3 2(t+2)(t+1).
\]
A natural choice for $\sA$ would thus be an Ulrich line bundle on $Y$.
In terms of Chern classes (as a coherent sheaf on $X$), we should have:
\[
c_1(\sA)=0, \qquad c_2(\sA)=-H_X^2.
\]

Since an Ulrich line bundle on a cubic surface comes from a twisted cubic, 
we need to choose another twisted cubic $D$ in $Y$, 
construct a surjection $\sS \to \sO_Y(D)$ so that the induced map on $H^0$ is an isomorphism, and take the kernel to perform an elementary modification. 
To do this, from now on in this section, we assume
 that $Y$ is the blow-up of $\PP^2$ at the six points $p_1,\ldots,p_6$
 in general position and that the blow-down map $\pi : Y \to \PP^2$ is
 associated with the linear system $|\sO_Y(C)|$. Write $L$ for the
 class of a line
 in $\PP^2$ and denote by $E_1,\ldots,E_6$ the exceptional
 divisors of $\pi$, so that $C=\pi^*L$ and $H_Y=3C-E_1-\cdots - E_6$.

{ We will relate the pulled-back cotangent bundle of
  $\PP^2$ to the sheaf $\sR$ defined in \eqref{defR},
  by proving that $\sR(1) \simeq \pi^*(\Omega_{\PP^2}(2))$, see
 Lemma \ref{iso}}.

 \begin{lem}
   Let $Z = \{p_1,p_2,p_3\}$. Then we have:
   \[
     0 \to \sO_{\PP^2}(-2) \to \Omega_{\PP^2}(1) \to \sI_{Z/\PP^2}(1) \to 0.
   \]
 \end{lem}

 \begin{proof}
   By assumption $Z$ is contained in no line, hence
   by the Cayley-Bacharach property (see for instance \cite[Theorem
   5.1.1]{HL10}) there is a vector bundle $\sF$ of rank $2$
   fitting into:
   \[
     0 \to \sO_{\PP^2}(-2) \to \sF \to \sI_{Z/\PP^2}(1)
     \to 0.
   \]

   Note that $c_1(\sF)=-L$ and $c_2(\sF)=L^2$. By the above sequence
   $\rH^0(\sF)=0$ so $\sF$ is stable. But the only stable bundle on
   $\PP^2$ with $c_1(\sF)=-L$ and $c_2(\sF)=L^2$ is $\Omega_{\PP^2}(1)$.
 \end{proof}

Set $D=2 C-E_1-E_2-E_3$. This is a class of a twisted
cubic in $Y$ with:
\[
  D \cdot C=2.
\]

 \begin{lem} \label{iso}
   There is a surjection $\eta : \sR(1) \to \sO_Y(D)$ such that the
   induced map on global sections
   $\rH^0(\sR(1)) \to \rH^0(\sO_Y(D))$ is an isomorphism.
 \end{lem}

 \begin{proof}
   {From \eqref{morexp}, we recall the exact sequence:
   \[
     0 \to \sO_Y(-C) \to 3\sO_Y \to \sR(1) \to 0.
   \]
   Comparing this with the Euler sequence, since $C=\pi^*L$, we get
   $\sR(1) \simeq \pi^*(\Omega_{\PP^2}(2))$. {It follows from the
   projection formula that:
   \begin{equation}
     \label{H0R}
     \rH^0(\sR)=0.
   \end{equation}
   }

 {Next, we relate the pull-back to $Y$ of the ideal sheaf of $Z$ to $\sO_Y(D)$. 
To do this we note that, given a smooth point $p$ of a surface $T$, the blow-up
 $\mu : \tilde  T \to T$ at $p$ gives an exceptional divisor $E$ and an exact sequence:
 \begin{equation}
   \label{blowupT}
   0 \to \sO_{E}(-1) \to \mu^*(\sI_{p/T}) \to
   \sO_{\tilde T}(-E) \to 0.
 \end{equation}
 Indeed, after localizing at $p$, the Kozsul resolution of $p$ in
 $T$ is given by two linear equations $f_1$ and $f_2$ vanishing at $p$
 and reads:
 \[
   0 \to \sO_T \to 2\sO_T \to \sI_{p/T} \to 0.
 \]
Pulling back via $\mu$ we get:
 \[
   0 \to \sO_{\tilde T} \to 2\sO_{\tilde T} \to \mu^*(\sI_{p/T}) \to 0.
 \]
The pull-back of $f_1$ and $f_2$ vanish at $E$, so  $\sO_{\tilde T} \to 2\sO_{\tilde T}$ factors through
 $\sO_{\tilde T} \to \sO_{\tilde T}(E) \to 2\sO_{\tilde T}$, hence we get a
 commutative diagram:
 \[
   \xymatrix{
     0 \ar[r] & \sO_{\tilde T} \ar[d] \ar[r] & 2\sO_{\tilde T} \ar@{=}[d] \ar[r] &
     \mu^*(\sI_{p/T} \ar[d] ) \ar[r] & 0\\
     0 \ar[r] & \sO_{\tilde T}(E) \ar[r] & 2\sO_{\tilde T} \ar[r] &
     \sO_{\tilde T}(-E) \ar[r] & 0}
 \]
 By the snake lemma we get thus \eqref{blowupT}.
 Applying this subsequently to the blow-up of $\PP^2$ at $p_1$,
 $p_2$, $p_3$, we get the exact sequence:
 \begin{equation}
   \label{pi*IZ}
   0 \to \bigoplus_{i=1}^3\sO_{E_i}(-1) \to \pi^* (\sI_{Z/\PP^2}(2))
   \to \sO_Y(D) \to 0.
 \end{equation}}
} 
      
   By the previous lemma, we have
   \begin{equation}
     \label{Omega2}
     0 \to \sO_{\PP^2}(-1) \to \Omega_{\PP^2}(2) \to \sI_{Z/\PP^2}(2)
     \to 0
   \end{equation}
   and thus via $\pi^*$ an exact sequence:
   \[
     0 \to \sO_Y(-C) \to \sR(1) \to \pi^*(\sI_{Z/\PP^2}(2)) \to 0.
   \]
   Composing $\sR(1) \to \pi^*(\sI_{Z/\PP^2}(2))$ with the surjection
   appearing in \eqref{pi*IZ}, we get the following:
   \[
     0 \to \sO_Y(-C+E_1+E_2+E_3) \to \sR(1) \to \sO_Y(D) \to 0.
   \]
   The map on global sections  $\rH^0(\sR(1)) \to \rH^0(\sO_Y(D))$ is
   induced by the map $\rH^0(\Omega_{\PP^2}(2)) \to
   \rH^0(\sI_{Z/\PP^2}(2))$ arising from \eqref{Omega2} and as such it
   is an isomorphism since $\rH^*(\sO_{\PP^2}(-1))=0$.
 \end{proof}

 Given the class of a twisted cubic $C$ in $Y$, we observe that
 $C^\rt = 2H_Y-C$ is also the class of a twisted cubic. We denote:
 \[
   C^\rt = 2H_Y-C.
 \]

 This notation is justified by the fact that $\sI_{C^\rt/Y}$ is
 presented by the transpose matrix $M^\rt$ of $M$. We have:
 \[
   C^\rt \cdot D =    C \cdot D^\rt = 4.
 \]
 
 \begin{lem}
   There is a surjection $\zeta : \sS \to \sO_Y(D)$ inducing an
   isomorphism:
   \[
     \rH^0(X,\sS) \to \rH^0(\sO_Y(D)).
   \]
 \end{lem}

 \begin{proof}
 According to the previous lemma, we have $\eta : \sR(1) \to \sO_Y(D)$
 inducing an isomorphism on global sections.
 We would like
 to use Lemma \ref{SonY} to lift $\eta$ to a surjection $\sS|_Y \to \sO_Y(D)$ and compose this
 lift with the restriction $\sS \to
 \sS|_Y$ preserving the isomorphism on global sections.

 So in the notation of Lemma \ref{SonY} we first lift $\eta$ to
 $\ker(\xi)$. To do this, we apply $\Hom_Y(-,\sO_Y(D))$ to
 \eqref{kerxi} and get:
  \[
   \cdots \to \Hom_Y(\ker(\xi),\sO_Y(D)) \to \Hom_Y(\sR(1),\sO_Y(D)) \to
   2\rH^1(\sO_Y(C+D-H_Y)) \to \cdots
 \]
 
 Now, $C+D-H_Y=E_4+E_5+E_6$ so $\rH^1(\sO_Y(C+D-H_Y))=0$. Therefore $\eta$
 lifts to $\hat \eta : \ker(\xi) \to \sO_Y(D)$. Note that by
 \eqref{kerxi} the map $\sR(1) \to \ker(\xi)$ induces an isomorphism on
 global sections, so $\hat \eta$ gives an isomorphism
 $\rH^0(\ker(\xi)) \simeq \rH^0(\sO_Y(D))$. 

 Next, write:
 \[
   0 \to \ker(\xi) \to \sS|_Y \to \sR \to 0,
 \]
 and apply $\Hom_Y(-,\sO_Y(D))$. We get an exact sequence:
 \[
   \cdots \to \Hom_Y(\sS|_Y,\sO_Y(D)) \to \Hom_Y(\ker(\xi),\sO_Y(D))
   \to \Ext^1_Y(\sR,\sO_Y(D)) \to \cdots 
 \]
 So $\hat \eta$ lifts to $\sS|_Y \to \sO_Y(D)$ if we prove
 $\Ext^1_Y(\sR,\sO_Y(D))=0$. To do it, write again the defining sequence
 of $\sR$ as:
 \[
   0 \to \sR \to 3\sO_Y \to \sO_Y(C^\rt) \to 0.
 \]
 Applying $\Hom_Y(-,\sO_Y(D))$ to this sequence we get:
 \[
   \cdots \to 3\rH^1(\sO_Y(D)) \to \Ext^1_Y(\sR,\sO_Y(D)) \to
   \rH^2(\sO_Y(C+D-2H_Y)) {\color{red}\to} \cdots 
 \]
 Now, $\sO_Y(D)$ is Ulrich so $\rH^1(\sO_Y(D))=0$, and $H_Y-C-D=-E_4-E_5-E_6$:
 \[
   \rh^2(\sO_Y(C+D-2H_Y))=\rh^0(\sO_Y(H_Y-C-D))=0.
 \]

 This provides a lift $\tilde \eta : \sS|_Y \to \sO_Y(D)$ of $\hat \eta$ and again
 $\ker(\xi) \mono \sS|_Y$ induces an isomorphism on global sections,
 hence so does $\tilde \eta$.

 Finally we define $\zeta : \sS \to \sO_Y(D)$ as composition of the
 restriction $\sS \to \sS|_Y$ and $\tilde \eta$. Since
{ $H^*(\sS(-1))=H^*(\sS(-2))=0$ by Proposition \ref{syzygy of twisted}}, tensoring the Koszul resolution
 \eqref{koszulY} by $\sS$ we see that $\sS \to \sS|_Y$ induces an
 isomorphism on global sections. Therefore, so does $\zeta$ and the
 lemma is proved.
 \end{proof}

 Consider $D^\rt = 2H_Y-D$ and $\sG_{D^\rt} = \ker(3\sO_X \to \sO_Y(D))$. 
 Let $\sE = \ker(\zeta)$, so we have:
 \begin{equation}
   \label{ESD}
   0 \to \sE \to \sS \to \sO_Y(D) \to 0.   
 \end{equation}

 \begin{lem}  The sheaf $\sE$ is simple and
   has a Jordan-H\"older filtration :
   \begin{equation}
     \label{sE}
     0 \to \sG_{D^\rt} \to \sE \to \sG_C \to 0.     
   \end{equation}
   Also, we have:
   \[
     \sE^\vee \simeq \sS^\vee, \qquad \rp_{\sE}(t) = \ru(t-1),
     \qquad \ch(\sE)=2\gamma, \qquad \rH^*(\sE(-t))=0,
     \qquad \mbox{for $t=0,1,2$}.
   \]
 \end{lem}

 \begin{proof}
   The sheaves $\sG_{D^\rt}$ and $\sG_C$ are stable by
   Lemma \ref{hat S is stable} and the reduced Hilbert polynomial of
   both of them is $\ru(t-1)$. Also, they are not isomorphic so
   $\Hom_X(\sG_{D^\rt},\sG_C)=\Hom_X(\sG_C,\sG_{D^\rt})=0$.

   {The sheaf $\sE$ is
   semistable and has reduced Hilbert polynomial $\ru(t-1)$ as soon
   as it fits in \eqref{sE}.
   Also, $\sE$ is simple if
   this sequence is non-split.
   To see this, first apply $\Hom_X(-,\sG_{D^\rt})$ to \eqref{sE} and note
   that the identity of $\sG_{D^\rt}$ is carried to the non-split extension represented
   by $\sE$. Hence, using that $\sG_{D^\rt}$ is simple and $\Hom_X(\sG_C,\sG_{D^\rt})=0$ we get 
   $\Hom_X(\sE,\sG_{D^\rt})=0$.
   Then, apply $\Hom_X(\sG_C,-)$ to \eqref{sE} and use that
   $\Hom_X(\sG_C,\sG_{D^\rt})=0$ and that $\sG_C$ is simple to get
   that $\Hom_X(\sG_C,\sE)$ is at most $1$-dimensional.
   Finally, apply $\Hom_X(-, \sE)$ to \eqref{sE} to deduce that
   $\End_X(\sE)$ is included $\Hom_X(\sG_C,\sE)$ so that $\End_X(\sE)$
   is $1$-dimensional and $\sE$ is simple.}
    
   The Chern character of $\sE$ is also computed from \eqref{sE} --
   recall the notation ${ \gamma=3-H_X^2+\frac {1}{4} H^4_X}$ from the Introduction.
   Moreover, by Lemma \ref{hat S is stable}, we get $\rH^*(\sE(-t))=0$,
   for $t=0,1,2$ as well by \eqref{sE}.

   Summing up, it suffices to prove that $\sE$ fits in \eqref{sE}
   and that this sequence is non-split. To do it, use the previous
   lemma to show that the evaluation of global sections gives an exact
   commutative diagram:
   \[
     \xymatrix@-2ex{
       & 0 \ar[d] & 0\ar[d] & & \\
       0 \ar[r] & \sG_{D^\rt} \ar[r] \ar[d] & \sE \ar[r] \ar[d]
       & \sG_C \ar[r] \ar@{=}[d] & 0 \\
       0 \ar[r] &  3\sO_X \ar[r] \ar[d] & \sS \ar[r] \ar[d]  & \sG_C \ar[r] & 0 \\
       &  \sO_Y(D) \ar@{=}[r] \ar[d] & \sO_Y(D) \ar[d] & & \\
       & 0 & 0 & & 
}
\]

We thus have \eqref{sE}. By contradiction, assume that it
splits as $\sE \simeq \sG_C \oplus \sG_{D^\rt}$.
Note that, since $\sS$ is locally free, \eqref{ESD} gives:
\[
  \sE^\vee \simeq \sS^\vee, \qquad  \sext^1_X(\sE,\sO_X) \simeq
  \sext_X^2 (\sO_Y(D), \sO_X) \simeq
  \sO_Y(D^\rt).
\]
On the other hand, if $\sE \simeq \sG_C \oplus \sG_{D^\rt}$
then by Lemma \ref{hat S is stable} we would have $\sE^\vee \simeq
6\sO_X$ and $\sext_X^1(\sE,\sO_X) \simeq \sO_Y(C) \oplus \sO_Y(D^\rt)$, which is not the case.
 \end{proof}

\section{Smoothing the modified sheaves} \label{section:smoothing}

 In the previous section we constructed a simple and semistable sheaf
 $\sE$ with $\rp_{\sE}(t) = u(t-1)$. In particular, the sheaf $\sE(1)$
 has the same reduced Hilbert polynomial as an Ulrich bundle
 $\sU$. However,  $\sE(1)$ itself cannot be Ulrich: for instance it is not locally free since $\sext_X^1(\sE,
 \sO_X) \simeq \sO_Y(D^\rt)$. 
The goal of this section is to show that $\sE(1)$ admits a flat deformation to an Ulrich bundle.

 \subsection{The Kuznetsov category}

 The bounded derived category $\rD(X)$ of coherent sheaves on $X$ has
 the semiorthogonal decomposition:
 \[
   \langle \Ku(X),\sO_X , \sO_X(1), \sO_X(2)\rangle,
 \]
 where $\Ku(X)$ is a K3 category. Indeed, $\Ku(X)$ equips with the K3-type Serre duality
 \[
 \Ext^i (\sF, \sG) \simeq \Ext^{2-i} (\sG, \sF)^{\vee}
 \]
 for any $\sF, \sG \in \Ku(X)$ \cite{Kuz04}.  We have:
 \[
   \rH^*(X,\sE)=\rH^*(X,\sE(-1))=\rH^*(X,\sE(-2))=0,
 \]
and therefore:
 \[
   \sE \in \Ku(X).
 \]
 Lemma \ref{hat S is stable} says that for a Cohen-Macaulay twisted
 cubic $C \subset X$ spanning an irreducible cubic surface we have
 that $\sG_C$ is stable and:
 \[
   \sG_C \in \Ku(X).
 \]
 We also know that $\sG_C$ represents a point of the
 Lehn-Lehn-Sorger-van Straten irreducible symplectic eightfold $\sZ$ and
 that $\sZ$ contains a Zariski-open dense subset $\sZ^\circ$ whose
 points are in bijection with the sheaves of the form $\sG_C$  \cite{LLSS17,LLMS18}.
 
 \begin{lem}
   We have $\rh^3(\sE(-3)) = \rh^4(\sE(-3)) = 3$, $\ext^i_X (\sE,
   \sE)=0$ for $i \ge 3$ and:
   \[
     \ext^1_X (\sE, \sE)=26, \qquad 
     \ext^2_X (\sE, \sE)=1.
   \]
 \end{lem}

 \begin{proof}
 First note that $\rh^3 (\sE(-3)) = \rh^2(\sO_Y(D-3H_Y))=3$ 
 since $\sO_Y(D)$ is an Ulrich line bundle on a cubic surface
 $Y$. We also have $\rh^4(\sE(-3))= \rh^4 (\sS(-3))=3$
 since $\chi(\sS(-3)) = 6 \rp_\sS (-3) = 3$ and $\rh^i(\sS(-3))=0$ for $i<4$.  

 Recall that $\sE$ is a simple sheaf, $\ch(\sE) = 2\gamma$, and that $\sE$ lies in
 $\Ku(X)$. Since $\Ku(X)$ is a K3 category, we have
 $\ext^2_X (\sE, \sE)=\hom_X(\sE,\sE)=1$ and $\ext^i_X (\sE, \sE)=0$
 for $i \ge 3$.
 Now the equality $\ext^1_X (\sE, \sE)=26$ follows from Riemann-Roch: 
 
\[
{\chi(\sE \otimes \sE^{\vee}) = \left[ \left( 36 - 24H_X^2 + 10H_X^4 \right) \left( 1 + \frac{3}{2}H_X + \frac{5}{4}H_X^2 + \frac{3}{4}H_X^3 + \frac{1}{3} H_X^4 \right) \right]_4 =  -24.}
\]

\end{proof}

\subsection{Deforming to Ulrich bundles}

We assume in this subsection that $X$ does not contain an integral surface of degree up to $3$ other than linear sections. In other words, $X$
does not contain a plane (equivalently, a quadric surface) nor a smooth cubic scroll, nor a cone over a rational normal cubic curve.

The content of our main result is that there is a smooth
connected quasi-projective variety $T^\circ$ of 
dimension $26$ and a sheaf $\sF$ on $T^\circ \times X$, flat over $T^\circ$,
together with a distinguished point $s_0 \in T^\circ$ such that $\sF_{s_0}
\simeq \sE$ and such that $\sF_s(1)$ is an Ulrich bundle on $X$, for all
$s$ in $T^\circ \setminus \{s_0\}$ -- here we write 
$\sF_s = \sF|_{\{s\} \times X}$ for all $s \in T^\circ$.
Stated in short form this gives the next result. 

\begin{thm} \label{theorem:no_plane}
If $Y$ is smooth, then the sheaf $\sE(1)$ deforms to an Ulrich bundle on $X$.
\end{thm}

\begin{proof} We divide the proof into several steps.
  \begin{step} \label{step:cohomology}
    Compute negative cohomology of $\sE$, i.e. $\rh^k(\sE(-t))$ for
    $t \gg 0$ and $k \in \{0,1,2,3\}$.
  \end{step}

  Let $C \subset Y \subset X$ be a twisted cubic with $Y$ smooth.
  The sheaf $\sG_C$ is stable and lies in $\Ku(X)$ by Lemma
  \ref{hat S is stable}. 
  We note that $\rh^1(\sO_Y(D+t H_Y))=0$ for $t \in \ZZ$, while
  ${\rh^2(\sO_Y(D-t H_Y))=0}$ for $t \le 1$ while:
  \[
    \rh^2(\sO_Y(D-t H_Y))=\frac 32 (t-1)(t-2), \qquad \mbox{for $t
      \ge 2$}.
  \]
  Also, $\rH^0(\sE)=\rH^1(\sE)=0$ since the surjection \eqref{ESD} induces an
  isomorphism on global sections. This also implies that, since
  $\rH^0(\sO_Y(D)) \otimes 
  \rH^0(\sO_X(t))$ generates $\rH^0(\sO_Y(D+t H_Y))$ for all $t \ge 0$,
  the map $\rH^0(\sS(t)) \to \rH^0(\sO_Y(D+t H_Y))$ induced by
  \eqref{ESD} is surjective. Since
  $\rH^1(\sS(t))=0$ for all $t \in \ZZ$, we obtain $\rH^1(\sE(t))=0$ for $t
  \in \ZZ$.
  By \eqref{ESD} we have:
  \[\left\{
    \begin{array}{ll}
     \rh^0(\sE(-t))=0, & \mbox{$t \ge 0$}, \\
      \rh^1(\sE(t))=0, & \mbox{$t \in \ZZ$},\\
      \rh^2(\sE(t))=0, & \mbox{$t \in \ZZ$},\\
     \rh^3(\sE(-t))=\frac 32 (t-1)(t-2), & \mbox{$t \ge 2$}.
    \end{array}\right.
   \]
  
   \begin{step}  \label{step:deformation}
     Argue that $\sE$ is unobstructed.
   \end{step}

   This follows from the argument of \cite[\S 31]{BLMNPS19}, which
   applies to the sheaf $\sE$ as it is simple and lies in
   $\Ku(X)$. 
   To sketch this, recall that the framework is based on a combination
   of Mukai's unobstructedness 
   theorem \cite{mukai:symplectic} and Buchweitz-Flenner's approach
   to the deformation theory of $\sE$, see \cite{BF00,BF03}.
   To achieve this step, we use the proof of \cite[Theorem
   31.1]{BLMNPS19} which goes as follows. Let $\At(\sE) \in
   \Ext^1_X(\sE,\sE \otimes \Omega_X)$ be the Atiyah class of $\sE$.
   
   \begin{itemize}
   \item Via a standard use of the infinitesimal lifting
   criterion, one reduces to show that $\sE$ has a
   formally smooth deformation space.
   \item We show that the deformation space of $\sE$ is formally
     smooth by observing that $\sE$ extends over any square-zero
     thickening of $X$, conditionally to
     the vanishing of the product of the Atiyah class $\At(\sE)$ and the 
     Kodaira-Spencer class $\kappa$ of the thickening, see 
   \cite{huybrechts-thomas:deformation} -- note that this holds
   in arbitrary characteristic.
 \item We use \cite{KM09} in order to show $\kappa \cdot \At(\sE)
   = 0$. Indeed, in view of \cite[Theorem 4.3]{KM09}, this takes place if
   the trace $\tr(\kappa \cdot \At^2(\sE))$ vanishes as element of
   $\rH^3(\Omega_X)$.

 \item We use that $\tr(\kappa \cdot \At(\sE)^2)=2\kappa \cdot
   \ch(\sE)$, (cf. the proof of \cite[Theorem 31.1]{BLMNPS19}) and
   note that this vanishes as the Chern character 
   $\ch(\sE)$ remains algebraic under any deformation of $X$. It holds 
   as all components of $\ch(\sE)$ are multiples of powers of the hyperplane
   class, while $\kappa \cdot \ch(\sE)$ is the obstruction to
   algebraicity of $\ch(\sE)$ along the thickening of $X$ -- again,
   cf. the  proof of \cite[Theorem 31.1]{BLMNPS19}.
   
   Note that the assumption that that $\kk$ has characteristic other
   than $2$ is needed to use the formula $\tr(\kappa \cdot
   \At(\sE)^2)=2 \tr(\kappa \cdot \exp(\At(\sE)))=2\kappa \cdot
   \ch(\sE)$. 
   \end{itemize}

   According to the above deformation argument, there is a smooth quasi-projective scheme $T$ representing an open piece of the
   moduli space of simple sheaves over $X$ containing $\sE$. In other
   words, 
   there is a point $s_0 \in T$, together with a coherent sheaf
   $\sF$ over $T \times X$, such that $\sF_{s_0} \simeq \sE$, and the
   Zariski tangent space of $T$ at $s_0$ is identified with
   $\Ext^1_X(\sE,\sE)$. 
   Note that all the sheaves $\sF_s$ are simple.
   By the openness of semistability and torsion-freeness, there is a
   connected open dense subset $T_0 \subset T$, with $s_0
   \in T_0$, such that $\sF_s$ is simple, semistable and torsion-free
   for all $s \in T_0$.

   \begin{step} \label{step:hull}
     Compute the cohomology of the reflexive hull $\sF_s^{\vee \vee}$ of
     $\sF_s$ and of $\sF_s^{\vee \vee}/\sF_s$.
   \end{step}
   For $s \in T_0$, let us consider the reflexive hull
  $\sF_s^{\vee \vee}$ and the torsion sheaf $\sQ_s = \sF_s^{\vee \vee}/\sF_s$. Let us write the reflexive hull sequence:
  \begin{equation} \label{Fs**}
    0 \to \sF_s \to \sF_s^{\vee \vee} \to \sQ_s \to 0.    
  \end{equation}
  
  By the upper-semicontinuity of cohomology, there is a Zariski-open dense
  subset $s_0 \in T_1$
  of $T_0$ such that for all $s \in T_1$ we have:
  \[
    \rH^*(\sF_s) = \rH^*(\sF_s(-1)) = \rH^*(\sF_s(-2)) = 0.
  \]
 In particular, for $s \in T_1$ and $t \ge 0$:
  \[
      \rh^k(\sF_s(-t))=0, \qquad
      \mbox{for $k \le 2$}.
   \]
   By Lemma \ref{reflexive}, since $\sF_s$ is torsion-free there is a
   polynomial $q_2 \in \QQ[t]$, with $\deg(q_2) \le 2$, such that
   $\rh^3(\sF_s(-t))=q_2(t)$ for $t \gg 0$. By semicontinuity,
   there is a Zariski-open dense subset $T_2$ of $T_1$, with $s_0 \in T_2$, such that for all
   $s \in T_2$ we have $q_2(t) \le \frac 32 (t-1)(t-2)$.
   
  Since $\codim(\sQ_s) \ge 2$, we get $\rH^k(\sQ_s(t))=0$ for each $k
  \ge 3$ and $t \in \ZZ$.  Using Lemma \ref{reflexive}, we get the vanishing
  $\rH^0(\sF_s^{\vee \vee}(-t))=\rH^1(\sF_s^{\vee \vee}(-t))=0$ for $t
  \gg 0$, and the existence of polynomials $q_0,q_1 \in \QQ[t]$ with $\deg(q_k) \le k$
  such that $\rh^{k+2}(\sF_s^{\vee \vee}(-t))=q_k(t)$ for $t \gg 0$.
  By \eqref{Fs**}, for $t \gg 0$ and $k \ne 2$ we have:
  \[\left\{
      \begin{array}{l}
        \rh^2(\sQ_s(-t))=q_2(t)-q_1(t)+q_0, \\
        \rh^k(\sQ_s(-t))= 0.
    \end{array}\right.
  \]

  Next, we use again the local-global spectral sequence
  \[
    \Ext_X^{p+q}(\sQ_s,\sO_X(t-3)) \Leftarrow \rH^p(\sext_X^q(\sQ_s,\sO_X(t-3))) = E_2^{p,q}.
  \]
  Via Serre vanishing for $t \gg 0$ and Serre duality this gives:
  \begin{align}
    \nonumber &\rh^2(\sQ_s(-t)) = \rh^0(\sext_X^2(\sQ_s,\sO_X(t-3))), \\
    \label{extQ} &\sext_X^k(\sQ_s,\sO_X)=0, \quad \mbox{for $k \ne 2$}.
  \end{align}

  Assume $\sQ_s \ne 0$. By the above discussion, $\sQ_s$ is a non-zero reflexive sheaf supported
  on a codimension 2 subvariety $Y_s$ of $X$, in which case 
  $h^2 (\sQ_s (-t))$ must agree with a polynomial function of degree
  $2$ of
  $t$ for $t \gg 0$.
  Hence the
  sheaf $\hat \sQ_s = \sext_X^2(\sQ_s,\sO_X(-3))$ supported on $Y_s$ satisfies:
  \begin{equation}
    \label{boundchi}
    \chi(\hat \sQ_s(t)) = h^2(\sQ_s(-t)) = q_2(t)-q_1(t)+q_0 \le \frac 32 (t-1)(t-2),
    \quad \deg(\chi(\hat \sQ_s(t)))=2.
  \end{equation}

  Note that \eqref{extQ} and \cite[Proposition 1.1.10]{HL10} imply 
  $\sext_X^k(\sF_s^{\vee \vee},\sO_X) = 0$ for $k \ge 3$ and therefore,
  via \eqref{Fs**},
  also $\sext_X^k(\sF_s,\sO_X)=0$ for $k \ge 3$.
  We prove along the way that:
  \begin{equation}
    \label{local Ext2}
    \rH^1(\sext_X^2(\sF_s(t),\sO_X))= 0, \qquad \mbox{for $t \gg 0$.}
  \end{equation}
  Indeed, dualizing 
  \eqref{Fs**} and using \eqref{extQ} we get an epimorphism, for $t
  \in \ZZ$:
  \[
    \sext_X^2(\sF_s^{\vee \vee}(t),\sO_X) \epi  \sext_X^2(\sF_s(t),\sO_X).
  \]
  By \cite[Proposition 1.1.10]{HL10}, if the sheaf
  $\sext_X^2(\sF_s^{\vee \vee},\sO_X)$ is non-zero then it is
  supported on a zero-dimensional subscheme of $X$, hence the same
  happens to $\sext_X^2(\sF_s,\sO_X(-t))$ by the above epimorphism.
  Therefore $\rH^1(\sext_X^2(\sF_s,\sO_X(-t)))=0$ for $t \gg 0$.
  
  \begin{step} \label{step:extension}
    Show that, if $\sF_s$ is not reflexive, then it is an extension of
    sheaves coming from $\sZ^\circ$.
  \end{step}
  We have proved that, if $\sF_s$ is not reflexive, the support $Y_s$
  of $\hat \sQ_s$ is a surface of degree at most 
  $3$. But then, since $X$ contains no integral surface of degree up to $3$
  other than complete intersections, the reduced structure of each
  primary component of $Y_s$ must be a
  cubic surface contained in $X$, and hence $Y_s$ itself must be a cubic surface.
  So the open subset $T_2
  \subset T_1$ provides a family of cubic surfaces $\sY \to T_2$ whose
  fibre over $T_2$ is the cubic surface $Y_s$, where $Y_{s_0}=Y$
  is smooth.
  Since smoothness is an open
  condition, there is a Zariski-open dense subset $T_3$ of $T_2$, with
  $s_0 \in T_3$, such
  that $Y_s$ is smooth for all $s \in T_3$.

  It follows again by \eqref{boundchi} that $\sQ_s$ is reflexive of rank
  $1$ over $Y_s$, i.e. $\sQ_s$ is a line bundle on $Y_s$ since $Y_s$ is smooth.
  Hence, we have a family of sheaves $\{\sQ_s \mid s \in T_3 \}$ where
  $\sQ_s$ is a line bundle over $Y_s$ and $\sQ_{s_0}\simeq \sO_Y(D)$. Since the Picard
  group of $Y_s$ is discrete, this family must be locally constant. In
  other words, for each $s \in T_3$ there is a divisor class $D_s$ on $Y_s$
  corresponding to a twisted
  cubic contained in $Y_s$ such that $\sQ_s \simeq \sO_{Y_s}(D_s)$ and $D_{s_0} \equiv D$.

  Since $\rH^0(\sF_s)=\rH^1(\sF_s)=0$, the evaluation map of global
  sections $3\sO_X \to \sO_{Y_s}(D_s)$ lifts to a non-zero map
  $\beta_s : 3 \sO_X \to \sF_s^{\vee \vee}$. The snake lemma yields an exact
  sequence:
  \[
    0 \to \ker(\beta_s) \to \sG_{D_s^\rt} \to \sF_s \to \coker(\beta_s) \to 0.
  \]
  Since the sheaves $\sF_s$ and $\sG_{D_s^\rt}$ share the same reduced
  Hilbert polynomial, with $\sF_s$ semistable and $\sG_{D_s^\rt}$
  stable, we must have $\ker(\beta_s)=0$. By semistability of $\sF_s$,
  we note that $\sD_s=\coker(\beta_s)$ is
  torsion-free, since otherwise the reduced Hilbert polynomial of the
  torsion-free part of $\sD_s$ would be strictly smaller than
  $\rp_{\sD_s} = \rp_{\sF_s}$.  

  Therefore, we have a flat family of sheaves over $T_3$ whose fibre
  over $s$ is $\sD_s$, with $\sD_{s_0} \simeq \sG_C$. Hence, for
  all $s \in T_3$, the sheaf $\sD_{s}$ corresponds to a point of the
  open part $\sZ^\circ$ of the 
  Lehn-Lehn-Sorger-van Straten eightfold, i.e. $\sD_{s_0} \simeq \sG_{C_s}$ for some twisted cubic $C_s \subset X$.
  We take a further Zariski-open dense subset $T_4$ of $T_3$ such that 
  $C_s$ is Cohen-Macaulay and spans a smooth cubic surface, for all $s
  \in T_4$.

  Summing up, in a Zariski-open neighbourhood $T_4$ of $s_0$, dense in
  $T$, the
  hypothesis $\sQ_s \ne 0$ for $s \in T_4$ implies the existence of
  twisted cubics 
  $D_s$, $C_s$ in $X$ such that $\sF_s$ fits into:
  \[
    0 \to \sG_{D_s^\rt} \to \sF_s \to \sG_{C_s} \to 0,
  \]
  where the twisted cubic $C_s$ is Cohen-Macaulay, so that $\sG_{C_s}$
  lies in $\Ku(X)$. Therefore the sheaves $\sG_{C_s}$ and
  $\sG_{D_s^\rt}$ correspond uniquely to well-defined points of $\sZ^\circ$.

  \begin{step} \label{step:ulrich}
    Conclude that $\sF_s(1)$ is Ulrich for generic $s \in T$.
  \end{step}

  We compute the dimension of the family $\sW$ of
  sheaves $\sF_s$ fitting into extensions as in the previous display.
  Indeed, $\sW$ is
  equipped with a regular map $\sW \to \sZ^\circ \times \sZ^\circ$
  defined by $\sF_s \mapsto (\sG_{D_s^\rt},\sG_{C_s})$,
  whose fibre is $\PP(\Ext^1_X(\sG_{C_s}, \sG_{D_s^\rt}))$. Since
  $D^\rt=D_{s_0}^\rt$ and $C=C_{s_0}$ are contained in $Y$ and satisfy
  $C \cdot D^\rt=4$, $C \cdot C^\rt=5$, we have $C \not \equiv D^\rt$ so $\sG_{D_{s_0}^\rt} \not \simeq \sG_{C_{s_0}}$.
  Therefore $\sG_{D_s^\rt} \not \simeq \sG_{C_s}$ for all
  $s$ in a Zariski-open dense subset $T_5 \subset T_4$, with $s_0 \in T_5$.
  Since $\sG_{D_s^\rt}$ and $\sG_{C_s}$ lie in $\Ku(X)$ and
  represent stable non-isomorphic sheaves, we have:
  \[
    \hom_X(\sG_{C_s}, \sG_{D_s^\rt})=0, \qquad
    \ext^2_X(\sG_{C_s}, \sG_{D_s^\rt}) \simeq 
    \hom_X(\sG_{D_s^\rt}, \sG_{C_s})=0.
  \]
  Also, $\ext^k_X(\sG_{C_s}, \sG_{D_s^\rt})=0$ for $k \ge 3$ hence :
  \[
    \ext_X^1(\sG_{C_s}, \sG_{D_s^\rt})=-\chi(\sG_{C_s}, \sG_{D_s^\rt}) = 6.
  \]
  Therefore the fibre of $\sW \to \sZ^\circ \times \sZ^\circ$ is $5$-dimensional and:
  \[
    \dim(\sW)=2 \cdot \dim(\sZ^\circ) + \ext^1_X (\sG_{C_s}, \sG_{D_s^\rt}) -1 = 21.
  \]

  So there is a Zariski-open dense
  subset $T_6 \subset T_5$ with $s_0 \in T_6$, such that $\sQ_s=0$ for
  all $s \in T_6 \setminus \{s_0\}$. Hence $\sF_s$ is reflexive for
  all $s \in T_6 \setminus \{s_0\}$.
  Then  $\sF_s^\vee$ is also semistable and shares
  the same reduced Hilbert polynomial as $\sF_s$, hence we have:
  \[
    \rh^4(\sF_s(-3))=\ext^4_X(\sO_X(3),\sF_s) = \rh^0(\sF_s^\vee)=0.
  \]
  Since $\rh^k(\sF_s(-3))=0$ for $k \le 2$, by Riemann-Roch we obtain
  $\rh^3(\sF_s(-3))=0$, i.e. $\rH^*(\sF_s(-3))=0$. Now we have proved
  that $\sF_s(1)$ is Ulrich for $s \in T_6 \setminus \{s_0\}$.

  Put $T^\circ=T_6$. We have proved that, for all $s \in T^\circ
  \setminus \{s_0\}$, the sheaf $\sF_s(1)$ is an Ulrich bundle of rank $6$.
\end{proof}

\subsection{Fourfolds containing planes or cubic scrolls}
Now we turn our attention to the case of smooth cubic fourfolds $X$
containing a surface of degree up to three, other than linear sections.

The goal is to prove our main theorem from the introduction, in other
words, we would like to extend Theorem \ref{theorem:no_plane} to these fourfolds.
Note that Steps \ref{step:cohomology}, \ref{step:deformation} and \ref{step:hull} of the proof of Theorem
\ref{theorem:no_plane} are still valid.
Also, the argument of Step \ref{step:ulrich} holds once
Step \ref{step:extension} is established. 
Summing up, it remains to work out Step
\ref{step:extension}. Recalling the base scheme $T_2$ introduced in Step
\ref{step:hull}, we are done as soon as we prove the
following result.

\begin{prop} \label{a linear section}
  There is a Zariski-open neighborhood of $s_0$ in $T_2^\circ$ such that,
  for all $s \in T_2^\circ$, the sheaf $\sQ_s=\sF_s^{\vee \vee}/\sF_s$
  is either zero or it is supported on a linear section surface of $X$.
\end{prop}

\begin{proof}
We proved in Step \ref{step:hull} that, for $s \in
T_2$, the sheaf $\sQ_s$ is zero or it is
a locally Cohen-Macaulay sheaf supported on a projective surface $Y_s\subset
X$ with $\deg(Y_s) \le 3$. Assuming $\sQ_s \ne 0$, we would like to
show that $Y_s$ does not contain any surface $Z$ other than linear
sections of $X$.
Passing to the purely two-dimensional part of the reduced structure of
each primary component of $Z$, we may assume without loss of
generality that $Z$ is integral, still of degree at most $3$ and not
a linear section: we must then seek a contradiction. 
The surface $Z$ is either a plane,
or a quadric surface, or a smooth cubic scroll, or a cone over a
rational normal cubic. The Hilbert polynomial of $\sO_Z$ is thus either
$r_1=(t+1)(t+2)/2$, $r_2=(t+1)^2$, or $r_3=(t+1)(3t+2)/2$, and $Z$ is locally a complete intersection in any case.

\smallskip
We denote by $\sH$ union of primary components of $\Hilb_r(X)$
containing integral subschemes $Z\subset X$ having Hilbert polynomial
$r$, with $r \in \{r_1, r_2,r_3\}$. 
Note that $\Hilb_{r_1}(X)$ is a finite reduced
scheme consisting of planes contained in $X$. For $r=r_2$ or $r=r_3$, a priori a surface in
$\Hilb_r(X)$ might be badly singular. However, we have the 
following claim. 

\begin{claim} \label{itsP2}
Each surface of $\Hilb_{r_2}(X)$ is a reduced quadric. For $r=r_3$, all the surfaces of $\sH$ are purely $2$-dimensional Cohen-Macaulay. For $r=r_2$ or $r=r_3$, each component of $\sH$ is a projective plane.
\end{claim}

\begin{proof}
  Take a surface $Z=Z_h$ in $\Hilb_{r_2}(X)$. If $Z$ is reduced,
  then $Z$ is a quadric. Otherwise, the reduced structure of a
  component of $Z$ must be a plane $L \subset X$. By computing 
  the Hilbert polynomial of $\sI_{L/Z}$, we see that this sheaf must be
  supported on a plane $L' \subset Z$ and have rank one over
  $L'$. 
  Hence its $\sO_{L'}$-torsion-free part is of the form
  $\sI_{B/L'}(b)$ for a subscheme $B \subset L'$ and some $b \in \ZZ$.
  Note that $\sI_{L/Z} \simeq \sI_{L/X}/\sI_{Z/X}$, so the surjection
  $3\sO_X(-1) \to \sI_{L/X}$ induces an epimorphism $3\sO_{L'}(-1) \to
  \sI_{B/L'}(b)$, whence $b \in \{-1,0\}$.

  Computing Hilbert polynomials and arguing that the leading term of the Hilbert polynomial of the possible
  $L'$-torsion part of $\sI_{L/Z}$ must be non-negative, we see that
  actually $b=-1$. This in turn implies $B=\emptyset$, i.e. $\sI_{L/Z}
  \simeq \sO_{L'}(-1)$.
  This says that $Z$ is a quadric surface. A direct computation shows
  that $Z$ must be reduced, for a cubic fourfold containing a
  non-reduced quadric surface is singular at least along a subscheme
  of length 4.
  
  All surfaces of a component of $\Hilb_{r_2}(X)$ are residual to the
  same plane in $X$ so each component of $\Hilb_{r_2}(X)$ is the
  projective plane of linear sections of $X$ containing a given plane.

  \medskip
  Now assume $r=r_3$ and let $Z=Z_h \subset X$ be an integral surface,
  so that $Z$ is a smooth cubic scroll or a cone over a rational
  normal cubic. We work roughly like in Proposition \ref{syzygy of
    twisted}. The linear span $\spanY$ of $Z$ is a $\PP^4$ that cuts $X$ along a cubic threefold $W$ and $\sI_{Z/W}(2)$ is an Ulrich sheaf of rank $1$ over $W$ so we have a presentation:
  \begin{equation}
    \label{againM}
    0 \to 3\sO_{\spanY}(-1) \xr{M} 3\sO_{\spanY} \to \sI_{Z/W}(2) \to 0.    
  \end{equation}

  Note that the threefold $W$ can have only finitely many singular
  points as if $W$ had a $1$-dimensional family of singular points
  then $X$ would be singular along the intersection of this family
  and a quadric in $V$.

  The idea is to prove that, on one hand, denoting by $\sN_{Z/X}$ the normal sheaf of $Z$ in $X$, we
  have $\rh^0(\sN_{Z/X})=2$. On the other hand, inspired by \cite[\S 4.1.2]{Has00}, we describe an explicit
  projective plane parametrizing elements of $\sH$ by proving that
  each global section of $\sI_{Z/W}(2)$ gives an element of $\sH$ and
  that all elements obtained this way are Cohen-Macaulay and indeed
  contained in $W$.
  
  Let us first accomplish the second task. By \eqref{againM}, the
  projectivization $P=\PP(\sI_{Z/W}(2))$ is embedded into $V \times 
  \PP^2 = \PP(3\sO_V)$ and the subscheme $P$ is
  cut in $V \times \PP^2$ by $3$ linear equations defined by the
  columns of $M$. Write $\pi$ and $\sigma$ for the projections to
  $V$ and $\PP^2$ from $V \times \PP^2$ and by $\fh$, $\fl$ the
  pull-back to $V \times \PP^2$ of the hyperplane divisors of $V$ and $\PP^2$ via $\pi$ and $\sigma$. Use the same letters upon
  restriction to $P$. From the Koszul resolution we obtain:
  \[
    0 \to \sO_{\PP^2 \times V}(-2\fh-3\fl) \to 3\sO_{\PP^2 \times
      V}(-\fh-2\fl) \to 3\sO_{\PP^2 \times V}(-\fl) \to
    \sO_{\PP^2 \times V}(\fh) \to \sO_P(\fh) \to 0.
  \]
  Taking $\sigma_*$, we get that the sheaf
  $\sV=\sigma_*(\sO_P(\fh))$ fits into:
  \[
    0 \to 3\sO_{\PP^2}(-1) \xr{N} 5\sO_{\PP^2} \to \sV \to 0.
  \]

  Observe that $\rH^0(\sO_{\PP^2}(1))$ is naturally identified with
  $\rH^0(\sI_{Z/W}(2))$. The choice of a line $\ell \subset
  \PP^2$ corresponds uniquely to surjection $\ell :
  \rH^0(\sO_{\PP^2}(1)) \epi 2\kk$ and thus
  to an epimorphism $3\sO_V \to
  2\sO_{\spanY}$. Composing with $M$, the line $\ell$ gives 
  uniquely a matrix $M_\ell : 3\sO_{\spanY}(-1) \to 2\sO_{\spanY}$. 
  
  We have $\PP(\sV) \simeq P$. Note that the map $\pi : P \to W$ is birational since $\sI_{Z/W}(2)$ has rank $1$ over $W$ and $\sO_W$ has the same Hilbert polynomial as $\sO_P(\fh)$. Therefore, $P$ is
  irreducible and thus $\sV$ is torsion-free.
  In particular, for any line $\ell \subset \PP^2$, setting $N_\ell$ for the restriction
  of $N$ to $\ell$, we get that $N_\ell$ is injective and yields, by restriction of $\pi$ to $\ell$: 
  \[
    \pi_\ell : \PP(\sV|_\ell) \to Z_\ell \subset W,
  \]
  where $Z_\ell=\im(\pi_\ell)$ is a surface in $W$.
  The scheme $\PP(\sV|_\ell)$ is equipped with two divisor
  classes inherited from $P$, which we still denote by $\fl$ and $\fh$.
  The surface $Z_\ell$ is the image of $\PP(\sV|_\ell)$ by
  the linear system $|\sO_{\PP(\sV|_\ell)}(\fh)|$.

  Now $\sV|_\ell  \simeq \coker(N_\ell)$ is
 of the form $\sO_\ell(a_1) \oplus \sO_\ell(a_2) \oplus \sB$, where $0
  \le a_1 \le a_2 \le 3$, and $\sB$ is a torsion sheaf on $\ell$ of
  length $b$, with $a_1+a_2+b=3$.
	{The direct sum decomposition of $\coker(N_\ell)$ corresponds to a decomposition of $N_\ell$ into block-diagonal matrices.
	In turn, each block is classified by Kronecker-Weierstrass theory, see for instance according to \cite[Chapter 19.1]{burgisser-clausen-shokrollahi}.
	In view of this, once chosen homogeneous coordinates $(y_0:y_1)$ on
  $\ell$, a block of $N_\ell$ corresponding to $\sO_\ell(a)$ for some $a \ge 1$, after transposition, takes the form of the following matrix with $a$ rows and $a+1$ columns :
	\[
	      \begin{pmatrix}
					y_0 & y_1 & 0& \cdots \\
					0 & y_0 & y_1 & 0 \\
					\vdots & \ddots & \ddots & \ddots & 0\\
					0 & \cdots & 0 & y_0 & y_1
      \end{pmatrix}.
	\]
	On the other hand, $\sB$ is presented as cokernel of an injective matrix of size $b$, where $b$ is the length of $\sB$ 
	and the rank of the matrix drops at the points of the support of $\sB$.
	Up to choosing the coordinates $(y_0:y_1)$ suitably, since $b$ is at most $3$, we may assume that such support is contained in set-theoretically in 
	$\{(0:1),(1:0),(1:1)\}$. Also, non-reduced points in the support of $\sB$ are described in terms of Jordan blocks of $N_\ell$. 
	However, the presence of non-reduced points in the support of $\sB$ is incompatible with the smoothness of $X$. 
	Indeed the matrices $N_\ell$ would take one of the following normal forms, for $b=3$ :
    \[
      N_\ell^\rt = 
      \begin{pmatrix}
        y_0 & 0 & 0 & 0 & 0 \\
        0 & y_0 & 0 & 0 & 0   \\
        0 & 0 & y_0 & 0 & 0
      \end{pmatrix}, \qquad
	      N_\ell^\rt = 
      \begin{pmatrix}
        y_0 & y_1 & 0 & 0 & 0 \\
        0 & y_0 & 0 & 0 & 0   \\
        0 & 0 & y_0 & 0 & 0
      \end{pmatrix}, \qquad
			      N_\ell^\rt = 
      \begin{pmatrix}
        y_0 & y_1 & 0 & 0 & 0 \\
        0 & y_0 & y_1 & 0 & 0   \\
        0 & 0 & y_0 & 0 & 0
      \end{pmatrix},
    \]
		or for $b=2$ :
    \[
      N_\ell^\rt = 
      \begin{pmatrix}
        y_0 & y_1 & 0 & 0 & 0 \\
        0 & 0 & y_0 & 0 & 0    \\
        0 & 0 & 0 & y_0 & 0
      \end{pmatrix}, \qquad
      N_\ell^\rt = 
      \begin{pmatrix}
        y_0 & y_1 & 0 & 0 & 0 \\
        0 & 0 & y_0 & y_1 & 0    \\
        0 & 0 & 0 & y_0 & 0
      \end{pmatrix}.			
	\]
	One checks that any cubic fourfold $X$ containing a surface arising from a matrix of this form has at least one singular
	point. Indeed, the first, second and fourth cases above do not
        even give rise to surfaces in $\PP^5$, while in the third and the
        fifth cases the surface contains respectively a triple plane
        or a double plane and this forces $X$ to be singular.

	Summing up, after removing the cases forbidden by the smoothness of $X$, 
  in a
  suitable basis of $\rH^0(\sV)$ and $\rH^1(\sV(-1))$ and coordinates $(y_0:y_1)$ on
  $\ell$, the possibilities for $N_\ell$ are:}
  \begin{itemize}
  \item $(a_1,a_2,b)=(1,2,0)$: $Z_\ell$ is a smooth cubic scroll and:
    \[
      N_\ell^\rt = 
      \begin{pmatrix}
        y_0 & y_1 & 0 & 0 & 0 \\
        0 & 0 & y_0 & y_1 & 0   \\
        0 & 0 & 0 & y_0 & y_1
      \end{pmatrix}.
    \]
  \item $(a_1,a_2,b)=(0,3,0)$: $Z_\ell$ is a cone over a rational normal
    cubic curve and:
        \[
          N_\ell^\rt = 
          \begin{pmatrix}
            y_0 & y_1 & 0 & 0 & 0 \\
            0 & y_0 & y_1 & 0 & 0   \\
            0 & 0 & y_0 & y_1 & 0 
          \end{pmatrix}.
        \]
  \item $(a_1,a_2,b)=(1,1,1)$: $Z_\ell$ is the union of a $\PP^2$ and a
    smooth quadric meeting along a line.
        \[
          N_\ell^\rt = 
          \begin{pmatrix}
            y_0 & y_1 & 0 & 0 & 0 \\
            0 & 0 & y_0 & y_1 & 0    \\
            0 & 0 & 0 & 0 & y_0 
          \end{pmatrix}.
        \]
  \item $(a_1,a_2,b)=(0,2,1)$: $Z_\ell$ is the cone over the union of a
    smooth conic and a line meeting at a single point, spanning a
    $\PP^3 \subset V$ and having apex at a point outside $V$.
    \[
      N_\ell^\rt = 
      \begin{pmatrix}
        y_0 & y_1 & 0 & 0 & 0 \\
        0 & y_0 & y_1 & 0 & 0   \\
        0 & 0 & 0 & y_0 & 0 
      \end{pmatrix}.
    \]
  \item $(a_1,a_2,b)=(0,1,2)$: $Z_\ell$ is the cone over the union of a
    line and reducible conic, meeting at a length-two subscheme of the
    line, spanning a
    $\PP^3 \subset V$. The apex of the cone is a point outside $V$.
    \[
      N_\ell^\rt = 
      \begin{pmatrix}
        y_0 & y_1 & 0 & 0 & 0 \\
        0 & 0 & y_0 & 0 & 0    \\
        0 & 0 & 0 & y_1 & 0 
      \end{pmatrix}.
    \]
  \item $(a_1,a_2,b)=(0,0,3)$: $Z_\ell$ is a cone over a non-colinear
    subscheme of length $3$ in $\PP^2 \subset V$, having a skew $\PP^1
    \subset V$ as apex.
    \[
      N_\ell^\rt = 
      \begin{pmatrix}
        y_0 & 0 & 0 & 0 & 0 \\
        0 & y_1 & 0 & 0 & 0    \\
        0 & 0 & y_0-y_1 & 0 & 0 
      \end{pmatrix}.
    \]
  \end{itemize}

  In all these cases the resulting subscheme $Z_\ell$ lies in $\sH$
  and has projective dimension $2$ with a Hilbert-Burch resolution
  given by $M_\ell^\rt$.
  Then the dual plane parametrizing lines $\ell
  \subset \PP^2$ describes an explicit projective plane of
  arithmetically Cohen-Macaulay surfaces in $\sH$.

  \medskip
  Finally we have to show that $\rh^0(\sN_{Z/X})=2$. We have an exact
  sequence:
  \[
    0 \to \sO_X(-1) \to \sI_{Z/X} \to \sI_{Z/W} \to 0.
  \]
  Applying $\shom_X(-,\sO_Z)$ we get:
  \[
    0 \to \sN_{Z/W} \to \sN_{Z/X} \to \sO_Z(1) \xr{\delta} \sext_X^1(\sI_{Z/W},\sO_Z) \to \sext_X^1(\sI_{Z/X},\sO_Z) \to 0.
  \]

  Since the surfaces $Z$ under consideration are locally complete
  intersection in $X$, we get that $\sext_X^1(\sI_{Z/X},\sO_Z) \simeq
  \sext_X^2(\sO_Z,\sO_Z)$ is the determinant of the normal bundle
  $\sN_{Z/X}$ and is thus identified with the line bundle $\sN_{Z/W}(1)$.
  On the other hand, using \eqref{againM} we see that the sheaf
  $\sext_X^1(\sI_{Z/W},\sO_Z)$ fails to be locally free of rank $1$ at
  the subscheme $\Upsilon \subset W$ defined by the $2$-minors of $M$. Since
  $\Upsilon$ is contained in (though sometimes not equal to) the singular locus of $W$, we have
  $\dim(\Upsilon) = 0$ so the resolution of $\Upsilon$ is
  obtained by the Gulliksen-Negard complex {(see \cite{GN72})}:
  \[
    0 \to \sO_V(-6) \to 9\sO_V(-4) \to 16\sO_V(-3) \to 9\sO_V(-2) \to \sI_{\Upsilon/V} \to 0.
  \]

  Thus $\Upsilon$ has length $6$ and $\rH^0(\sI_{\Upsilon/V}(1))=0$, which
  in turn implies $\rH^0(\sI_{\Upsilon/Z}(1))=0$.
  Therefore, $\ker(\delta)\subset \sI_{\Upsilon/Z}(1)$ gives
  $\rH^0(\ker(\delta))=0$. In turn we get $\rH^0(\sN_{Z/X}) \simeq
  \rH^0(\sN_{Z/W})$ so it only remains to show
  $\rh^0(\sN_{Z/W})=2$. To get this, since $Z$ and $W$ are locally complete
  intersection, we may use adjunction to the effect that $\sN_{Z/W}
  \simeq \shom_W( \omega_W,\omega_Z)/\sO_W$. Therefore, using
  $\omega_W \simeq \sO_W(-2)$ and restricting
  \eqref{againM} to $W$ we get:
  \[
    0 \to \sI_{Z/W}(-1) \to 3\sO_W(-1) \to 2\sO_W \to \sN_{Z/W} \to 0.
  \]
  Taking cohomology we obtain $\rh^0(\sN_{Z/W})=2$ as desired.
\end{proof}

Write $\sZ \subset X \times \sH$ for the tautological surface. For each
point $h \in \sH$, we denote by $Z_h=\sZ \cap X \times \{h\}$ the
corresponding surface. Consider $\sX = X \times T_2 \times \sH$ and write
$\pi_{1,2}$, $\pi_{1,3}$ and $\pi_{2,3}$ for the projections from
$\sX$ onto $X
\times T_2$, $X \times \sH$ and $T_2 \times \sH$.
We have the following claim.

\begin{claim} \label{component}
  For any $(s,h) \in T_2 \times \sH$, the surfaces
  $Z_h$ and $Y_s$ share a component if and only if:
  \[
    \rH^2(\sext_X^1(\sF_s(t),\sO_{Z_h})) \ne 0, \qquad \mbox{for $t \gg 0$}.
  \]
\end{claim}

\begin{proof}
  Take $(s,h) \in T_2 \times \sH$ and set $Z=Z_h$.
  Since $\sF_s$ is torsion-free
  and $\sF_s^{\vee \vee}$ is reflexive, we have, for any coherent sheaf
  $\sB$ on $X$:
  \begin{equation}
    \label{van1}
    \sext^q_X(\sF_s^{\vee \vee},\sB)= \sext^{q+1}_X(\sF_s,\sB)=0
    \quad \mbox{for $q \ge 3$},
  \end{equation}
  and, for $q \in \{1,2\}$: 
  \begin{equation}
    \label{van2}
    \dim(\sext^{q}_X(\sF_s^{\vee \vee},\sB)) \le 2-q, \qquad 
    \dim(\sext^q_X(\sF_s,\sB)) \le 3-q.
  \end{equation}
  Indeed, this follows from  \cite[Proposition 1.1.10]{HL10} if $\sB$ is
  locally free. Then, \eqref{van1} and \eqref{van2} hold for an
  arbitrary coherent sheaf $\sB$ as we see by applying
  $\shom_X(\sF_s,-)$ and $\shom_X(\sF_s^{\vee \vee},-)$ to a finite
  locally resolution of $\sB$ and using that \eqref{van1} and \eqref{van2}
  hold for the terms of the resolution.

  Applying $\shom_X(-,\sO_{Z})$ to \eqref{Fs**} we get, for
  $q \ge 1$:
  \[
    \cdots \to \sext^q_X(\sF_s^{\vee \vee},\sO_{Z}) \to
    \sext^q_X(\sF_s,\sO_{Z}) \to
    \sext^{q+1}_X(\sQ_s,\sO_{Z}) \to 
    \sext^{q+1}_X(\sF_s^{\vee \vee},\sO_{Z}) \to \cdots
  \]
  We deduce that $\sext^q_X(\sF_s,\sO_Z)=0$ for $q \ge 3$ and
  \[
    \dim(\sext^1_X(\sF_s,\sO_{Z})) = 2 ~ \Longleftrightarrow ~
    \dim(\sext^2_X(\sQ_s,\sO_{Z})) = 2.
  \]
  Therefore
  \begin{equation}
    \label{3}
    \rH^p(\sext^1_X(\sF_s(t),\sO_{Z})) = 0 ~ \mbox{ for $p \ge 3$ and
      all $t \in \ZZ$},
  \end{equation}
  and
  \begin{equation}
    \label{dim2}
    \rH^2(\sext^1_X(\sF_s(t),\sO_{Z})) \ne 0 ~ \mbox{ for $t \gg 0$}
    ~ \Longleftrightarrow ~
    \dim(\sext^2_X(\sQ_s,\sO_{Z})) = 2.
  \end{equation}

By Claim \ref{itsP2}, we may assume that $Z$ is a locally
Cohen-Macaulay in $X$. 
{ So, since $\sO_Z$ is a locally Cohen-Macaulay $\sO_X$-module of
projective dimension $2$, by the Hilbert-Burch theorem, locally over
$X$ there exists a
matrix $M$ of size $p \times (p+1)$ whose $p$-minors
cut $Z$ as subscheme of $X$, namely we have a local presentation:
\[
  0 \to p\sO_X \to (p+1) \sO_X \to \sO_X \to \sO_Z \to 0.
\]
Applying $\shom_X(\sQ,-)$ to this sequence and using
$\sext_X^k(\sQ_s,\sO_X)=0$ for $k \ne 2$ we get that the sheaf $\sext_X^2(\sQ_s,\sO_{Z})$ is
locally presented as cokernel of the rightmost map in:
\begin{equation}
  \label{localpresentation}
  0 \to p\hat \sQ_s \xr{\shom(\hat \sQ_s,M)} (p+1) \hat \sQ_s
  \xr{\shom(\hat \sQ_s,\wedge^p M)} \hat \sQ_s
\end{equation}}

Now the $p$-minors of $M$ vanish on an irreducible component of $Y_s$
if and only if such component also lies in $Z$, in which case 
\eqref{localpresentation} shows that the support of $\sext_X^2(\sQ_s,\sO_{Z})$ is
the whole component.
Conversely, if $Y_s$ and $Z$ share no irreducible component so that
the $p$-minors do not vanish identically on any component of $Y_s$, then again by
\eqref{localpresentation} the sheaf $\sext_X^2(\sQ_s,\sO_{Z})$ is
supported along a closed subset of $Z$ having
dimension at most $1$. 
This shows that $\dim(\sext_X^2(\sQ_s,\sO_{Z}))=2$ if and only if
$\sQ_s$ and $Z$ share a common component. Together with \eqref{dim2},
this finishes the proof.
\end{proof}

\begin{claim} \label{relative ext}
  For $t \in \ZZ$, put
  $\sB = \sext^1_\sX(\pi_{12}^*(\sF(-t)),\pi_{13}^*(\sO_\sZ))$
  and $\sP=\rR^2\pi_{23 *}(\sB)$.
  For $(s,h) \in T_2 \times \sH$, we have $\sP_{(s,h)} \ne 0$ for
  $t \gg 0$ if and
  only if $Z_h$ and $Y_s$ have a common component.
\end{claim}

\begin{proof}
  Since $\sF$ and $\sO_{\sZ}$ are flat over $T_2 \times \sH$, we have an identification $\sB_{(s,h)} \simeq  \sext^1_X(\sF_s(-t),\sO_{Z_h})$
  for all $t \in \ZZ$ and $(s,h)  \in T_2 \times \sH$.  
  By the vanishing results of the previous
  paragraph and using the flattening stratification for $\sB$ over $T_2 \times \sH$ and working over each stratum, we
  get $\rR^{p}  \pi_{23 *} (\sB)=0$ for all $p \ge 3$ and $t \in
  \ZZ$ so via base-change we obtain, for all $(s,h) \in T_2 \times
  \sH$, we have
  \[
    \sP_{(s,h)} \simeq \rR^2\pi_{23 *}(\sB)_{(s,h)} \simeq \rH^2
    \left(\sext^1_X(\sF_s(t),\sO_{Z_h})\right)
  \]  
  for all $t \in \ZZ$. The conclusion follows from Claim \ref{component}.
\end{proof}

Now we finish the proof of the proposition. {Indeed, by Claim
\ref{itsP2} for the
special point $s_0 \in T_2$, the surface $Y=Y_{s_0}$ shares no
component with any surface $Z_h$ for $h \in \sH$.
Indeed, since $Z_h$ and $Y$ are projective surfaces of
degree $3$ in $\PP^5$, if $Z_h$
contains $Y$ then $Z_h$ contains a further (possibly embedded) component of
dimension at most $1$, hence the surface  $Z_h$ would 
fail to be Cohen-Macaulay, a contradiction.}

Therefore, by Claim
\ref{relative ext} we have $\sP_{(s_0,h)}=0$ for all $h \in \sH$.
In other words, the support of $\sP$ is disjoint from $\{s_0\} \times
\sH$. Since $\sH$ is projective, the image of the support of $\sP$ in $T_2$
is thus a closed subset of $T_2$, disjoint from $s_0$.
Hence there exists an open neighborhood $T_2^\circ$ of $s_0$
disjoint from this subset.
Thus the support of
$\sP$ does not intersect $T_2^\circ \times \sH$. This implies that,
for all $s \in T_2^\circ$, if $\sQ_s$ is not zero then its support is
a surface $Y_s$ having degree at most $3$
and containing no surface of $\sH$ as a component, in other words
$Y_s$ must be a linear section of $X$. This completes the proof of
Proposition \ref{a linear section} and consequently of the main
theorem.
\end{proof}

\section{Stability of Ulrich bundles and the higher rank range} \label{k}

Our next goal is to prove Theorem \ref{main2}. We know that $\sE(1)$
deforms to a simple Ulrich bundle, and such Ulrich bundles constitute a family of dimension $26$.
We check that a sufficiently general deformation of $\sS$ is stable, which will
achieve the proof of the first statement.

\begin{lem}
  The sheaf $\sE(1)$ deforms flatly to a stable Ulrich bundle. 
\end{lem}

\begin{proof}
  We know that $\sE(1)$ deforms to an Ulrich vector bundle $\sU$.
  Also, $\sU$ is necessarily semistable and any element of the graded
  objects associated to a Jordan-H\"older filtration of $\sU$ must be
  a stable Ulrich bundle. Therefore it suffices to prove that, for a generic
  choice of $\sU$, we have $\Hom_X(\sU,\sB)=\Hom_X(\sB,\sU)=0$, where
  $\sB$ is an Ulrich bundle of rank $2$, and that $\sU$ is not the
  extension of two Ulrich bundles of rank $3$.

  If $X$ supports an Ulrich bundle $\sB$ of rank $2$, which is to say, if
  $X$ is a pfaffian cubic fourfold, then we have $c_2(\sB(-1))\cdot H_X^2=2$ and
  $c_2(\sB(-1))^2=6$ so $\chi(\sB,\sB)=2$. Since $\sB(-1)$ is stable
  and lies in
  $\Ku(X)$, this implies that $\Ext^1_X(\sB,\sB)=0$, therefore there are
  finitely many Ulrich bundles $\sB$ of rank $2$ on $X$. For each of
  them and any twisted cubic $C \subset X$ we have
  $\Hom_X(\sG_C,\sB(-1))=0=\Hom_X(\sB(-1),\sG_C)$ for these sheaves are stable,
  not isomorphic, and share the same reduced Hilbert polynomial.
  Thus  $\rH^4(X,\sB^\vee \otimes \sE(-2))\simeq \Hom_X(\sE,\sB(-1))^\vee  =0$ and
  $\rH^0(X,\sB^\vee \otimes \sE(1))  \simeq \Hom_X(\sB(-1),\sE) =0$ by
  \eqref{sE}. Therefore, 
  $\Hom_X(\sU,\sB)=0=\Hom_X(\sB,\sU)$ for a sufficiently general choice
  of $\sU$ by semicontinuity of cohomology.

  Next, consider Ulrich bundles $\sA$ and $\sB$ of rank $3$ on $X$ and
  assume that $\sU$ contains $\sA$ with $\sB \simeq \sU/\sA$.
  Note that the degeneracy locus of two sufficiently general global
  sections of $\sA$ is a smooth surface $A \subset X$ with $A\cdot
  H_X^2=12$ and $A \cdot A = 54$, which implies that
  $\chi(\sA,\sA)=0$. Similarly we have $\chi(\sB,\sB)=0$ and the
  relation $\ch(\sU)=\ch(\sA)+\ch(\sB)$ gives $\chi(\sA,\sB)=\frac 12
  \chi(\sU,\sU)=12$. It follows that the family of Ulrich bundles
  admitting a Jordan-H\"older filtration with graded object $\sA \oplus
  \sB$ has dimension 15. Hence a sufficiently general choice of $\sU$
  is stable.

  \medskip
  To prove the second statement of Theorem \ref{main2}, set $\sU_1$
  for a stable Ulrich bundle of rank $6$ obtained as before and assume that
  $X$ is general enough so that it contains no surface which is not
  homologous to 
  $m H_X^2$ for any $m \in \NN$ and that it supports 
  a Cartan bundle $\sU_2$ arising from \cite{IM14}. We have
  {$\ch(\sU_i(-1))=(i+1)\gamma$ for $i \in \{1,2\}$ (cf. \cite[Section 2]{KS20} for the computation of Chern classes of Ulrich bundles on a (very) general cubic fourfold)}. Also, $\sU_1$ and $\sU_2$ are
  stable. Indeed they are semistable and may be only
  destabilized by Ulrich bundles, which is impossible, for Ulrich bundles must 
  have rank $3m$ for some $m \in \NN$, with $m \ge 2$, when $X$ is
  very general according to
  \cite{KS20}.

  Take $i$, $j$ distinct in $\{1,2\}$. From the observations above 
  we deduce
  $\Hom_X(\sU_i,\sU_j)=0$ so,
  since $\sU_1(-1)$ and $\sU_2(-1)$ lie in
  $\Ku(X)$ we get $\Ext^{\ell}_X(\sU_i,\sU_j)=0$ unless $\ell=1$. By Riemann-Roch we have
  $\ext^1_X(\sU_i,\sU_j)=36$.

  Consider an integer $k > 1$. We may find integers $k_1,k_2 \in \NN$ such
  that $k=2k_1+3k_2$.
  Then we construct a deformation $\sU$ of $k_1\sU_1 \oplus k_2 \sU_2$
  which is a stable Ulrich bundle.
  Indeed, we first consider a simple sheaf $\sU_0$ as an extension of
  $k_1\sU_1$ and $k_2 \sU_2$, which is possible for
  $\chi(k_1\sU_1,k_2\sU_2)=-36k_1k_2$.
  The sheaf $\sU_0$ has a smooth deformation space of dimension
  $6k^2+2$ and we take $\sU$ to be a generic element of this space.
  Now, the summands of the graded
  object arising from a Jordan-H\"older filtration of $\sU$ must be
   stable Ulrich bundles and thus, after twisting by $\sO_X(-1)$, they
   must have Chern character $h_1\gamma,
   \ldots, h_r\gamma$ for
   some $r,h_1,\ldots,h_r \in \NN^*$ with $h_1+\cdots+h_r=k$.
   But the dimension of the family of sheaves admitting such a filtration is:
   \[
     6\left(k^2-\sum_{1 \le i < j \le r} h_i h_j\right) + r+1 < 6k^2+2,
   \]
   as one can see by looking at the Jordan-H\"older filtration as $r$
   iterated extensions, the inequality being valid for all $r > 1$.

   This proves that for any $k > 1$ there is a stable Ulrich bundle
   $\sU$ on
   $X$ with $\ch(\sU(-1))=k\gamma$ which is a generic flat deformation
   of $k_1\sU_1 \oplus k_2 \sU_2$. The moduli space of stable sheaves
   $\rM_X(k\gamma)$ is smooth and symplectic at the points
   corresponding to these sheaves.
 \end{proof}

\newcommand{\etalchar}[1]{$^{#1}$}
\providecommand{\bysame}{\leavevmode\hbox to3em{\hrulefill}\thinspace}
\providecommand{\MR}{\relax\ifhmode\unskip\space\fi MR }
\providecommand{\MRhref}[2]{%
  \href{http://www.ams.org/mathscinet-getitem?mr=#1}{#2}
}
\providecommand{\href}[2]{#2}


\providecommand{\bysame}{\leavevmode\hbox to3em{\hrulefill}\thinspace}
\providecommand{\MR}{\relax\ifhmode\unskip\space\fi MR }
\providecommand{\MRhref}[2]{%
  \href{http://www.ams.org/mathscinet-getitem?mr=#1}{#2}
}
\providecommand{\href}[2]{#2}
\begin{thebibliography}{}

\end{thebibliography}


\begin{thebibliography}{LLSvS17}

\bibitem[BCS97]{burgisser-clausen-shokrollahi}
P.~B{\"u}rgisser, M.~Clausen, and M.~A. Shokrollahi, \emph{Algebraic complexity
  theory}, Grundlehren der Mathematischen Wissenschaften, vol. 315,
  Springer-Verlag, Berlin, 1997, With the collaboration of Thomas Lickteig.

\bibitem[Bea00]{Bea00}
A.~Beauville, \emph{Determinantal hypersurfaces}, Michigan Math. J. \textbf{48}
  (2000), 39--64.

\bibitem[Bea02]{Bea02}
\bysame, \emph{Vector bundles on the cubic threefold}, Symposium in Honor of C.
  H. Clemens, Contemp. Math., vol. 312, Amer. Math. Soc., Providence, RI, 2002,
  pp.~71--86.

\bibitem[BEH87]{BEH87}
R.-O. Buchweitz, D.~Eisenbud, and J.~Herzog, \emph{Cohen-{M}acaulay modules on
  quadrics}, Singularities, representation of algebras, and vector bundles
  (Lambrecht, 1985), Lecture Notes in Mathematics, vol. 1273, Springer, Berlin,
  1987, pp.~58--116.

\bibitem[BES17]{BES17}
M.~Bl\"aser, D.~Eisenbud, and F.-O. Schreyer, \emph{Ulrich complexity},
  Differential Geom. Appl. \textbf{55} (2017), 128--145.

\bibitem[BF00]{BF00}
R.-O. Buchweitz and H.~Flenner, \emph{The {A}tiyah-{C}hern character yields the
  semiregularity map as well as the infinitesimal {A}bel-{J}acobi map}, The
  arithmetic and geometry of algebraic cycles ({B}anff, {AB}, 1998), CRM Proc.
  Lecture Notes, vol.~24, Amer. Math. Soc., Providence, RI, 2000, pp.~33--46.

\bibitem[BF03]{BF03}
\bysame, \emph{A semiregularity map for modules and applications to
  deformations}, Compositio Math. \textbf{137} (2003), no.~2, 135--210.

\bibitem[BF11]{BF11}
M.~C. Brambilla and D.~Faenzi, \emph{Moduli spaces of rank-2 {ACM} bundles on
  prime {F}ano threefolds}, Michigan Math. J. \textbf{60} (2011), 113--148.

\bibitem[BH89]{BH89}
J.~Backelin and J.~Herzog, \emph{On {U}lrich-modules over hypersurface ring},
  Commutative Algebra (Berkeley, 1987), Math. Sci. Res. Inst. Publ., vol.~15,
  Springer, New York, 1989, pp.~63--68.

\bibitem[BLM{\etalchar{+}}]{BLMNPS19}
A.~Bayer, M.~Lahoz, E.~Macr\`i, H.~Nuer, A.~Perry, and P.~Stellari,
  \emph{Stability conditions in familes}, \texttt{ar{X}iv:1902.08184}.

\bibitem[CH11]{CH11}
M.~Casanellas and R.~Hartshorne, \emph{A{CM} bundles on cubic surfaces}, J.
  Eur. Math. Soc. (JEMS) \textbf{13} (2011), no.~3, 709--731. \MR{2781930}

\bibitem[CH12]{CH12}
\bysame, \emph{Stable {U}lrich bundles}, Internat. J. Math. \textbf{23} (2012),
  1250083.

\bibitem[Eis80]{Eis80}
D.~Eisenbud, \emph{Homological algebra on a complete intersection}, Trans.
  Amer. Math. Soc. \textbf{260} (1980), 35--64.

\bibitem[ES03]{ESW03}
D.~Eisenbud and F.-O. Schreyer, \emph{Resultants and {C}how forms via exterior
  syzygies}, J. Amer. Math. Soc. \textbf{16} (2003), 537--579, with an appendix
  by J. Weyman.

\bibitem[ES11]{ES11}
\bysame, \emph{Boij-{S}\"oderberg theory}, Combinatorial aspects of commutative
  algebra and algebraic geometry, Abel Symp., vol.~6, Springer, Berlin, 2011,
  pp.~35--48.

\bibitem[Fae19]{Fae19}
D.~Faenzi, \emph{Ulrich bundles on ${K3}$ surfaces}, Algebra Number Theory
  \textbf{13} (2019), no.~6, 1443--1454.

\bibitem[FPL21]{FPL17}
D.~Faenzi and J.~Pons-Llopis, \emph{{The Cohen-Macaulay representation type of
  arithmetically Cohen-Macaulay varieties}}, {Épijournal de Géométrie
  Algébrique} \textbf{{5}} (2021).

\bibitem[GN72]{GN72}
T.~Gulliksen and O.~Negard, \emph{Un complex r\'esolvant pour certains ideaux
  determinantiels}, C. R. Acad. Sc. Paris ({A}) \textbf{274} (1972), 16--18.

\bibitem[Has00]{Has00}
B.~Hassett, \emph{Special cubic fourfolds}, Compositio Math. \textbf{120}
  (2000), 1--23.

\bibitem[HL10]{HL10}
D.~Huybrechts and M.~Lehn, \emph{The geometry of moduli spaces of sheaves},
  Cambridge {U}niversity {P}ress, 2010.

\bibitem[HT10]{huybrechts-thomas:deformation}
D.~Huybrechts and R.~P. Thomas, \emph{Deformation-obstruction theory for
  complexes via {A}tiyah and {K}odaira-{S}pencer classes}, Math. Ann.
  \textbf{346} (2010), no.~3, 545--569. \MR{2578562}

\bibitem[IM14]{IM14}
A.~Iliev and L.~Manivel, \emph{On cubic hypersurfaces of dimensions $7$ and
  $8$}, Proc. London Math. Soc. \textbf{108} (2014), 517--540.

\bibitem[KM09]{KM09}
A.~Kuznetsov and D.~Markushevich, \emph{Symplectic structures on moduli spaces
  of sheaves via the {A}tiyah class}, J. Geom. Phys. \textbf{59} (2009),
  843--860.

\bibitem[KS20]{KS20}
Y.~Kim and F.-O. Schreyer, \emph{An explicit matrix factorization of cubic
  hypersurfaces of small dimensions}, J. Pure Appl. Algebra \textbf{224}
  (2020), 106346.

\bibitem[Kuz04]{Kuz04}
A.~Kuznetsov, \emph{Derived category of a cubic threefold and the variety
  $v_{14}$}, Proc. Steklov Inst. Math. \textbf{246} (2004), 171--194.

\bibitem[LLMS18]{LLMS18}
M.~Lahoz, M.~Lehn, E.~Macr\`i, and P.~Stellari, \emph{Generalized twisted
  cubics on a cubic fourfold as a moduli space of stable objects}, J. Math.
  Pures Appl. \textbf{114} (2018), no.~9, 85--117.

\bibitem[LLSvS17]{LLSS17}
C.~Lehn, M.~Lehn, C.~Sorger, and D.~van Straten, \emph{Twisted cubics on cubic
  fourfolds}, J. Reine Angew. Math \textbf{731} (2017), 87--128.

\bibitem[LMS15]{LMS15}
M.~Lahoz, E.~Macr\`i, and P.~Stellari, \emph{{A}rithmetically
  {C}ohen-{M}acaulay bundles on cubic threefolds}, Algebr. Geom. \textbf{2}
  (2015), no.~2, 231--269.

\bibitem[LPZ20]{li-pertusi-zhao}
C.~Li, L.~Pertusi, and X.~Zhao, \emph{Twisted cubics on cubic fourfolds and
  stability conditions}, arXiv math.AG/1802.01134, 2020.

\bibitem[Man19]{Man19}
L.~Manivel, \emph{Ulrich and a{CM} bundles from invariant theory}, Comm.
  Algebra \textbf{47} (2019), 706--718.

\bibitem[Muk84]{mukai:symplectic}
S.~Mukai, \emph{Symplectic structure of the moduli space of sheaves on an
  abelian or {$K3$} surface}, Invent. Math. \textbf{77} (1984), no.~1,
  101--116. \MR{85j:14016}

\bibitem[TY22]{TY22}
H.~L. Truong and H.~N. Yen, \emph{Stable {U}lrich bundles on cubic fourfolds},
  ar{X}iv math.AG/2206.05285, 2022.

\bibitem[Ulr84]{Ulr84}
B.~Ulrich, \emph{Gorenstein rings and modules with high numbers of generators},
  Math. Z. \textbf{188} (1984), 23--32.

\end{thebibliography}
\end{document}